\documentclass{jT}

\usepackage{amssymb,amsmath,latexsym,amsthm}%or any usual package
\usepackage{mathrsfs,amsxtra,graphicx}
\usepackage[english]{babel}
\usepackage{hyperref}
\hypersetup{colorlinks=true,urlcolor=blue,citecolor=blue,linkcolor=blue}
\usepackage{colonequals}

\usepackage[mathcal]{euscript}
\newcommand{\Gal}{\operatorname{Gal}}

\newcommand{\p}{\mathfrak{p}}
\newcommand{\Fp}{\mathbf{F}_{\mathfrak{p}}}

\newcommand{\Aut}{\operatorname{Aut}}
\newcommand{\Ql}{\mathbf{Q}_{\ell}}
\newcommand{\A}{\mathbf{A}}
\newcommand{\Zl}{\mathbf{Z}_{\ell}}
\newcommand{\im}{\operatorname{im}}
\newcommand{\Spec}{\operatorname{Spec}}
\newcommand{\tr}{\operatorname{tr}}

\newcommand{\Z}{\mathbf{Z}}
\newcommand{\F}{\mathbf{F}}
\newcommand{\R}{\operatorname{\mathbf{R}}}

\newcommand{\Q}{\mathbf{Q}}

\newcommand{\PSL}{\operatorname{PSL}}

\newcommand{\GL}{\operatorname{GL}}
\newcommand{\SL}{\operatorname{SL}}
\newcommand{\Frob}{\operatorname{Frob}}

\newcommand{\PP}{\operatorname{\mathbf{P}}}
\newcommand{\opchar}{\operatorname{char}}

\newcommand{\C}{\mathbf{C}}

\newcommand{\Qbar}{\Q^{\textup{al}}}
\newcommand{\Kbar}{K^{\textup{al}}}
\newcommand{\Fbar}{F^{\textup{al}}}

\newcommand{\Res}{\operatorname{Res}}

\newcommand{\psmod}[1]{~(\textup{\text{mod}}~{#1})}

\newcommand{\frakp}{\mathfrak{p}}

\newcommand{\scrE}{\mathscr{E}}
\DeclareMathOperator{\height}{ht}

\DeclareMathOperator{\ord}{ord}
\newcommand{\abs}[1]{\lvert {#1} \rvert}

\newcommand{\Htwo}{\boldsymbol{\mathsf{H}}^2}

\theoremstyle{plain}
\newtheorem{thm}[equation]{Theorem}

\newtheorem{lem}[equation]{Lemma}
\newtheorem{defn}[equation]{Definition}
\newtheorem{cor}[equation]{Corollary}
\newtheorem{prop}[equation]{Proposition}

\theoremstyle{remark}
\newtheorem{rmk}[equation]{Remark}
\newtheorem{remark}[equation]{Remark}
\newtheorem{exm}[equation]{Example}

\numberwithin{equation}{subsection}

\newcommand{\defi}[1]{\textsf{#1}} 		

\def\sbtors{{}_{{\textup{tor}}}}

\DeclareMathOperator{\Area}{area}
\DeclareMathOperator{\area}{area}
\DeclareMathOperator{\len}{len}
\DeclareMathOperator{\bd}{bd}
\DeclareMathOperator{\lcm}{lcm}

\newenvironment{enumalph}
{\begin{enumerate}}
{\end{enumerate}}

\newenvironment{enumroman}
{\begin{enumerate}}
{\end{enumerate}}

\begin{document}

\title[A probabilistic local-global principle for torsion]{On a probabilistic local-global principle for torsion on elliptic curves}

%first author

\author{John Cullinan}
\address{Department of Mathematics, Bard College, Annandale-On-Hudson, NY 12504, USA}
\email{cullinan@bard.edu}
\urladdr{\url{http://faculty.bard.edu/cullinan/}}

\author{Meagan Kenney}
\address{Department of Mathematics, University of Minnesota, Minneapolis, MN 55455}
\email{kenn0699@umn.edu}

\author{John Voight}
\address{Department of Mathematics, Dartmouth College, 6188 Kemeny Hall, Hanover, NH 03755, USA}
\email{jvoight@gmail.com}
\urladdr{\url{http://www.math.dartmouth.edu/~jvoight/}}



\thanks{The authors would like to thank Robert Harron and Siman Wong for helpful conversations, Robert Lemke Oliver for comments, and Peter J.\ Cho, Keunyoung Jeong, Grant Molnar, Carl Pomerance, Edward Schaefer, David Zureick--Brown, and the anonymous referee for their feedback and corrections.  Voight was supported by an NSF CAREER Award (DMS-1151047) and a Simons Collaboration Grant (550029).}

\maketitle

\begin{resume}
Soit $m$ un entier positif et soit $E$ une courbe elliptique sur $\Q$ avec la propri\'et\'e que $m\mid\#E(\F_p)$ pour un ensemble de densit\'e $1$ de nombres premiers $p$. En nous appuyant sur les travaux de Katz et Harron--Snowden, nous \'etudions la probabilit\'e que $m \mid \#E(\Q)\sbtors$: nous trouvons qu'elle est non nulle pour tout $m \in \lbrace 1, 2, \dots , 10 \rbrace \cup \lbrace 12, 16 \rbrace$ et nous le calculons exactement quand $m \in \lbrace 1,2,3,4,5,7 \rbrace$. En compl\'ement, nous donnons un d\'ecompte asymptotique de courbes elliptiques avec une structure de niveau suppl\'ementaire lorsque la courbe modulaire param\'etrable r\'esulte du quotient par un groupe sans torsion de genre z\'ero.
\end{resume}

\begin{abstr}
Let $m$ be a positive integer and let $E$ be an elliptic curve over $\Q$ with the property that $m\mid\#E(\F_p)$ for a density $1$ set of primes $p$.  Building upon work of Katz and Harron--Snowden, we study the probability that $m \mid \#E(\Q)\sbtors$: we find it is nonzero for all $m \in \lbrace 1, 2, \dots, 10 \rbrace \cup \lbrace 12, 16 \rbrace$ and we compute it exactly when $m \in \lbrace 1,2,3,4,5,7 \rbrace$.  As a supplement, we give an asymptotic count of elliptic curves with extra level structure when the parametrizing modular curve arises from the quotient by a torsion-free group of genus zero.
\end{abstr}

\bigskip

\section{Introduction} \label{intro}

\subsection{Motivation}

Let $E$ be an elliptic curve over $\Q$ and let $E(\Q)\sbtors$ denote the torsion subgroup of its Mordell--Weil group.  If $p$ is a prime of good reduction for $E$ with $p \nmid \#E(\Q)\sbtors$, then we have an injection $E(\Q)\sbtors \hookrightarrow E(\F_p)$; consequently, if $m \mid \#E(\Q)\sbtors$ then $m \mid \#E(\F_p)$ for all but finitely many $p$.  The converse statement holds only \emph{up to isogeny}, by a result of Katz \cite[Theorem 2]{katz}: if $m \mid \#E(\F_p)$ for a set of primes $p$ of density $1$, then there exists an elliptic curve $E'$ over $\Q$ that is isogenous over $\Q$ to $E$ such that $m \mid \#E'(\Q)\sbtors$.  

We say $E$ \defi{locally has a subgroup of order $m$} if $m \mid \#E(\F_p)$ (equivalently, $m \mid \#E(\Q_p)\sbtors$) for a set of primes $p$ of density $1$.  With respect to the property of having a subgroup of order $m$, the result of Katz is then a local-global principle for \emph{isogeny classes} of elliptic curves.  In this paper, we consider a probabilistic refinement for the elliptic curves themselves: if $E$ locally has a subgroup of order $m$, what is the \emph{probability} that $E$ globally has a subgroup of order $m$?  

\subsection{Notation} \label{sec:notation}

Every elliptic curve $E$ over $\Q$ is defined by a unique equation of the form $y^2=f(x)=x^3+Ax+B$ with $A,B \in \Z$ such that $4A^3+27B^2 \neq 0$ and there is no prime $\ell$ such that $\ell^4 \mid A$ and $\ell^6 \mid B$.  Let $\scrE$ be the set of elliptic curves of this form, and define the \defi{height} of $E \in \scrE$ by 
\begin{equation} 
\height E \colonequals \max(\abs{4A^3},\abs{27B^2}). 
\end{equation}
For $H>0$, let $\scrE_{\leq H} \colonequals \{E \in \scrE : \height E \leq H\}$ be the finite set of elliptic curves of height at most $H$.  

For $m \in \Z_{\geq 1}$, let $\scrE_{m?}$ be the set of $E \in \scrE$ such that $E$ locally has a subgroup of order $m$.  In this notation, our goal is to study the probability
\begin{equation}  \label{eqn:limit_eqn}
P_m \colonequals \lim_{H \to \infty} \frac{\#\{E \in \scrE_{\leq H} : m \mid \#E(\Q)\sbtors\}}{\#\{E \in \scrE_{m?} \cap \scrE_{\leq H}\}}
\end{equation}
when this limit exists. 

\subsection{Results}

In view of the theorem of Mazur \cite{mazur} on rational torsion, we have $\scrE_{m?}$ nonempty if and only if $m \in \{1,2,\dots,10,12,16\}$.  Our main result is as follows.

\begin{thm} \label{mainresult}
For all $m \in \{1,2,\dots,10,12,16\}$, the probability $P_m$ defined in \textup{\eqref{eqn:limit_eqn}} exists and is nonzero.  Moreover, $P_m$ is effectively computable.  
\end{thm}

For $m=1$ we have vacuously $P_m=1$.  For $m=2$, we again have $P_m=1$ because if $E \in \scrE_{2?}$ then its defining cubic polynomial $f(x) \in \Z[x]$ has a root modulo $p$ for a set of primes of density $1$, so by the Chebotarev density theorem it has a root in $\Q$.  The cases where $m=3,4$ require special consideration and will be treated at the end of this section.  

For $m \geq 5$ in our list, our proof of Theorem \ref{mainresult} is carried out in the following way.  We show that $P_m$ can be expressed in terms of the number of points of bounded height on a finite list of explicitly given modular curves---reducing to the case where $m=\ell^n$ is a prime power, these curves arise from a careful study of the $\ell$-adic Galois representation, refining the above theorem of Katz (see \S\ref{sec:refinement}).  We then apply the principle of Lipschitz, counting points in a homogeneously expanding region, to count elliptic curves by height on these modular curves.  Taking the ratio, we then find a positive probability.  

To count elliptic curves by height, we establish a general result of potential independent interest: we extend work of Harron--Snowden \cite{hs}, who provide asymptotics for the number of elliptic curves of bounded height in a universal family, as follows.  Let $N \in \Z_{\geq 1}$ and let $G \leq \GL_2(\Z/N)$ be a subgroup with $\det(G)=(\Z/N)^\times$.  Let 
$\pi_N \colon \SL_2(\Z) \to \SL_2(\Z/N)$ be the projection map and let
\begin{equation}  \label{eqn:GammaG}
\Gamma_G \colonequals \pi_N^{-1}(G \cap \SL_2(\Z/N)) \leq \SL_2(\Z).
\end{equation}
Let $Y_G$ be the open modular curve obtained by taking the quotient of the upper half-plane by the action of $\Gamma_G$.  Let $\Gal_\Q \colonequals \Gal(\Qbar\,|\,\Q)$ and let 
\[ \overline{\rho}_{E,N} \colon \Gal_\Q \to \Aut(E[N](\Qbar)) \simeq \GL_2(\Z/N) \] 
be the Galois representation on the $N$-torsion subgroup of $E$.  We write $\overline{\rho}_{E,N}(\Gal_\Q) \lesssim G$ to mean that the image of $\overline{\rho}_{E,N}$ is conjugate in $\GL_2(\Z/N)$ to a subgroup of $G$.

\begin{thm} \label{thm:hsupgrade}
Let $G \leq \GL_2(\Z/N)$ be such that $\det G=(\Z/N)^\times$.  Suppose that $\Gamma_G$ is torsion free (in particular, $-1 \not\in\Gamma_G$) and that $Y_G$ has genus zero and no irregular cusps.  Let
\begin{equation}
d(G) \colonequals \tfrac{1}{2}[\PSL_2(\Z):\Gamma_G] = \tfrac{1}{4}[\SL_2(\Z):\Gamma_G].
\end{equation}
Then $d(G) \in \Z_{\geq 1}$, and there exists an effectively computable $c(G) \in \R_{\geq 0}$ such that
\begin{equation} 
\#\{E \in \scrE_{\leq H} : \overline{\rho}_{E,N}(\Gal_\Q) \lesssim G\} = c(G) H^{1/d(G)} + O(H^{1/e(G)})
\end{equation}
as $H \to \infty$, where $e(G)=2d(G)$.
\end{thm}

In particular, this theorem applies to the groups $G$ that arise in the proof of Theorem \ref{mainresult}.  Moreover, it allows us to count elliptic curves with (marked) torsion of size at least $5$; dealing with the remaining few cases separately, we have the following corollary.  

\begin{cor} \label{thm:hsupgrade_cor}
For each $T$ in Table \textup{\ref{tab:yup}}, we have 
\begin{equation} \label{eqn:EET}
\#\{E \in \scrE_{\leq H} : E(\Q)\sbtors \simeq T\} = c(T) H^{1/d(T)} + O(H^{1/e(T)}). 
\end{equation}
\end{cor}

In view of Table \ref{tab:yup}, the count of curves $E \in \scrE_{\leq H}$ such that $E(\Q)\sbtors$ merely \emph{contains} a subgroup isomorphic to $T$ has the same asymptotic as the count in \eqref{eqn:EET}.

\begin{equation} \label{tab:yup}\addtocounter{equation}{1} \notag
\begin{gathered}
{\renewcommand{\arraystretch}{1.1}
\begin{tabular}{c|c|c||c|c|c} 
$T$ & $1/d(T)$ & $1/e(T)$ & $T$ & $1/d(T)$ & $1/e(T)$ \\[0.1ex]
\hline
\hline
\{0\} & 5/6 & 1/2 & $\Z/9$, $\Z/10$ & 1/18 & 1/36 \\
$\Z/2$ & 1/2 & 1/3 & $\Z/12$ & 1/24 & 1/48 \\
$\Z/3$ & 1/3 & 1/4 & $\Z/2 \times \Z/2$ & 1/3 & 1/6 \\
$\Z/4$ & 1/4 & 1/6 & $\Z/2 \times \Z/4$ & 1/6 & 1/12 \\
$\Z/5$, $\Z/6$ & 1/6 & 1/12 & $\Z/2 \times \Z/6$ & 1/12 & 1/24 \\ 
$\Z/7$, $\Z/8$ & 1/12 & 1/24 & $\Z/2 \times \Z/8$ & 1/24 & 1/48
\end{tabular}} \\
\text{Table \ref{tab:yup}: Asymptotic count of elliptic curves with designated torsion}
\end{gathered}
\end{equation}

Harron--Snowden \cite[Theorem 1.2]{hs} proved that $\#\{E \in \scrE_{\leq H} : E(\Q)\sbtors \simeq T\} \asymp H^{1/d(T)}$ for the groups $T$ in Table \ref{tab:yup}, and gave the power-saving asymptotic with explicit constant \cite[Theorem 5.6]{hs} for $\#T \leq 3$.  Indeed, there has been a recent spate of work on the topic of counting elliptic curves with certain level structure by height \cite{boggess-sankar,bruinnaj,cj,ppj,ed_carl_arxiv}; the theorem above provides an asymptotic in cases not handled by these other works.

We follow the strategy of Harron--Snowden in the proof of Theorem \ref{thm:hsupgrade}, again applying the Principle of Lipschitz.  The constant $c(G)$ is given by a product of an area of a compact region in the plane multiplied by a sieving factor that includes certain effectively computable local correction factors.  The square-root error term accounts for the boundary of the region.  The hypotheses of Theorem \ref{thm:hsupgrade} ensure that the moduli problem defined by $G$ is fine, so there is a universal elliptic curve over the associated moduli scheme.  (In fact, there are only finitely many torsion-free, genus zero congruence subgroups $\Gamma_G \leq \SL_2(\Z)$---a list first compiled by Sebbar \cite{sebbar}.)

\begin{rmk}
Although the above result suffices for our purposes, echoing Harron--Snowden \cite[\S 1.5]{hs}, it would be desirable to establish a statement generalizing Theorem \ref{thm:hsupgrade} to an arbitrary group $G$ with $\Gamma_G$ of genus zero.  See work of Ellenberg--Satriano--Zureick-Brown \cite[\S 4]{eszb} for a conjecture of Batyrev--Manin--Malle type which predicts an estimate for the number of rational points of bounded height on stacky curves.
\end{rmk}

Returning to our main result, Theorem \ref{thm:hsupgrade} applies directly to the cases $m \geq 5$: tallying degrees $d(G)$, it is then straightforward to prove Theorem \ref{mainresult}.  In fact, we show that even before tallying the degrees $d(G)$, we know they are all equal for curves arising from ``isogenous'' moduli problems, as follows (Theorem \ref{thm:greenbergup-inart}).

\begin{thm} \label{thm:greenbergup}
Let $\varphi \colon E \to E'$ be an isogeny of elliptic curves over $\Q$.  Let $N \in \Z_{\geq 1}$, let $G \colonequals \overline{\rho}_{E,N}(\Gal_\Q) \leq \GL_2(\Z/N)$ and similarly $G'$ for $E'$.  Then the associated modular curves $Y_G$ and $Y_{G'}$ are isomorphic over $\Q$, and $d(G)=d(G')$.
\end{thm}

The invariance of the index of the adelic Galois representation under isogeny was proven by Greenberg \cite[Proposition 2.1.1]{greenberg} using a beautiful but very different argument; we view the isomorphism of modular curves as a refinement.   

Finally, carrying this out this strategy with an explicit calculation yields $P_5$ and $P_7$ in section \ref{sect57}.

\begin{thm} \label{57probthm}
We have $P_5 = 25/34 \approx 73.5\%$ and $P_7  = 4/(4+\sqrt{7}) \approx 60.2\%$.
\end{thm}

For $m=5$ and $m=7$, and more generally, our investigation reveals that the curves with torsion have \emph{smaller height} relative to their counterparts with just locally a subgroup of order $m$.

We now return to the remaining values $m=3,4$ are interesting in their own right and benefit from direct arguments, so we dig deeper.  Consider first the case $m=3$.  We first recall that every elliptic curve $E \in \mathscr{E}_{3?}$ either has a rational 3-torsion point or its quadratic twist by $-3$ does.  
With careful attention to local contributions at $3$, we find a matching growth rate for the quadratic twists, yielding the following result.

\begin{thm} \label{3probthm}
We have $P_3 =1/2$.
\end{thm}

So Theorem \ref{3probthm} says that among elliptic curves with $3 \mid \#E(\F_p)$ for almost all $p$, there are 50-50 odds that $3 \mid \#E(\Q)\sbtors$.  

When $m=4$, the situation is more complicated, due in part to the fact that $E$ can have $4 \mid \#E(\Q)\sbtors$ in two different ways.  We first show that having full $2$-torsion dominates having a point of order $4$ among elliptic curves in $\scrE_{4?}$ in the following sense.

\begin{prop} \label{prop:eweir}
We have $E \in \scrE_{4?}$ if and only if at least one of the following holds:
\begin{enumroman}
\item $E(\Q)[2] \simeq (\Z/2)^2$, or 
\item $E$ has a cyclic $4$-isogeny defined over $\Q$.
\end{enumroman}
\end{prop}

Proposition \ref{prop:eweir} can also be rephrased geometrically: if $E \in \scrE_{4?}$, then since $\scrE_{4?} \subseteq \scrE_{2?}$ the elliptic curve $E$ arises from a $\Q$-rational point on the classical modular curve $Y_0(2)=Y_1(2)$, and this point lifts to a $\Q$-rational point under at least one of the natural projection maps $Y(2) \to Y_0(2)$ or $Y_0(4) \to Y_0(2)$, each of degree $2$.

The fact that $4 \mid \#E(\F_p)$ for all good odd $p$ in case (ii) can be explained by a governing field that is biquadratic: for half of the good primes we have $E(\F_p)[2] \simeq (\Z/2)^2$ whereas for the  complementary half $E(\F_p)$ has an element of order $4$.  See Proposition \ref{4weierstrass} for details.  

We then count the number of elliptic curves in case (i) and (ii) with a direct argument: we find they have the same asymptotic rate of growth, with explicit constants.  Next, we show that among curves satisfying (ii), those with $4 \mid \#E(\Q)\sbtors$ are asymptotically negligible.  Therefore, $P_4$ is equal to the probability that $E$ belongs to case (i) among those curves belonging to (i) and (ii), giving the following result.

\begin{thm} \label{thm:P4}
There exists a constant $c_4\in\R_{>0}$ such that as $H \to \infty$, 
\begin{align*} 
&\#\{E \in \scrE_{\leq H} : \text{$E$ has a cyclic $4$-isogeny defined over $\Q$}\} \\
&\qquad\qquad = c_4 H^{1/3} + O(H^{1/6}). 
\end{align*}
Moreover, we have $c_4 \approx 0.9574$ and $P_4 \approx 27.2\%$ effectively computable.
\end{thm}

The exact value of $c_4$ is given in Proposition \ref{cyclic_isogeny_constants} and for $P_4$ in Proposition \ref{4prob}.  For both $m=3,4$, these theorems match experimental data (Remarks \ref{rmk:P3approx}, \ref{rmk:P4approx}).

\begin{rmk} \label{rmk:naiveheight0}
We choose to normalize our height function including the constants in the discriminant function, following Bhargava--Shankar \cite{BS}.  Alternatively, one can order the elliptic curves by defining
\[ \height'(E) \colonequals \max(\abs{A^3},\abs{B^2}) \]
(without the scalars $4,27$).  The probability for $m=3$ is again $1/2$ in this height; see Remark \ref{rmk:naiveheight} for the probability for $m=4$ computed in this way instead.  
\end{rmk}

\subsection{Organization}

Our paper is organized as follows.  In Section \ref{S1.6}, we collect relevant facts about Galois representations attached to elliptic curves as a way to reformulate our main question in terms of Galois image, refining work of Katz \cite{katz}.  With these images in hand, it then becomes a computation with universal curves to obtain the order of growth of curves in $\scrE_{m?}$ ordered by height.  In section \ref{sect57}, we use this to prove our main result for $m \geq 5$ and carry this out explicitly for $P_5,P_7$.  In section \ref{34div}, we treat the remaining cases $m=3,4$ in detail, computing the asymptotics and the relevant constants.

\section{Galois representations and divisibility} \label{S1.6}

In this section, we characterize the image of the Galois representation under the condition of local $m$-divisibility.  The main results of this section are Corollary \ref{cor:sumitup} and Theorem \ref{upgrade}: we bound the degree of an isogeny (guaranteed by the theorem of Katz \cite[Theorem~1]{katz}) from any elliptic curve $E$ with locally a subgroup of order $m$ to an elliptic curve $E'$ with a subgroup of order $m$.

\subsection{Setup} \label{setup_section}

We reset our notation, working in more generality to start.  Let $K$ be a number field with ring of integers $\Z_K$ and algebraic closure $\Kbar$.  Let $E$ be an elliptic curve over $K$ with origin $\infty \in E(K)$.  By a \defi{prime} of $K$ we mean a nonzero prime ideal $\frakp \subset \Z_K$, and we write $\Fp \colonequals \Z_K/\frakp$ for the residue field of $\frakp$; we say a prime $\frakp$ is \defi{good} (for $E$) if $\frakp$ is prime of good reduction for $E$.

Let $\ell \in \Z$ be prime and let $T_\ell E \colonequals \varprojlim_{n} E[\ell^n](\Kbar) \simeq \Z_\ell^2$ be the $\ell$-adic Tate module, writing $P=(P_n)_n \in T_\ell E$ with each $P_n \in E[\ell^n](\Kbar)$ satisfying $\ell P_n = P_{n-1}$.  The absolute Galois group $\Gal_K \colonequals \Gal(\Kbar\,|\,K)$ acts continuously on $T_\ell E$ giving a Galois representation 
\begin{equation} \label{eqn:tateyup}
\rho_{E,\ell}\colon \Gal_K \to \Aut_{\Z_\ell}(T_\ell E) \simeq \GL_2(\Z_\ell)
\end{equation}
with $\det \rho_{E,\ell} : \Gal_K \to \Z_\ell^\times$ equal to the $\ell$-adic cyclotomic character.  In the above, we follow the convention that matrices act on the left on column vectors.

  We write 
\begin{equation} 
\overline{\rho}_{E,\ell^n} : \Gal_K \to \Aut(E[\ell^n](\Kbar)) 
\end{equation}
for just the action on $E[\ell^n](\Kbar)$, alternatively obtained as the composition of $\rho_{E,\ell}$ with reduction modulo $\ell^n$.  We also define $V_\ell E \colonequals T_\ell E \otimes \Q_\ell$.

If $E'$ is another elliptic curve over $K$, by an isogeny $\varphi\colon E \to E'$ we mean an isogeny defined over $K$.  (If we have need to consider isogenies defined over an extension, we will indicate this explicitly.)

For a good prime $\frakp$ of $K$ coprime to $\ell$, we have
\begin{equation}
\#E(\Fp) = \det(1-\rho_{E,\ell}(\Frob_\p))
\end{equation}
where $\Frob_\p$ is the conjugacy class of the Frobenius automorphism at $\frakp$ in $\Gal_K$, and recall that the point counts $\#E(\Fp)$ are well-defined on the isogeny class of $E$.  Moreover, by the Chebotarev density theorem, the condition $\ell^n \mid \#E(\Fp)$ for a set of primes $\p$ of density $1$ is equivalent to the group-theoretic condition
\begin{equation} \label{eqn:detrhoE}
\det(1-\rho_{E,\ell}(\sigma)) \equiv 0 \pmod{\ell^n} 
\end{equation}
for all $\sigma \in \Gal_K$, and further $\ell^n \mid \#E(\Fp)$ for primes $\frakp$ in a set of density $1$ if and only if $\ell^n \mid \#E(\Fp)$ for all but finitely many $\frakp$.

\subsection{Galois images}

Both to motivate what follows and because we will make use of it, we begin with the following lemma.  

\begin{defn}
A basis $P_1,P_2$ for $T_\ell E$ is \defi{clean} if there exist $r,s \in \Z_{\geq 0}$ such that (in coordinates) $P_{1,r},P_{2,s}$ generate $E[\ell^\infty](K)$.
\end{defn}

Choosing generators, we see that $T_\ell E$ always has a clean basis.  Moreover, if $P_1,P_2$ is a clean basis, then the integers $r,s$ are unique and $E[\ell^\infty](K) \simeq \Z/\ell^r \times \Z/\ell^s$.

\begin{lem} \label{lem:cleantelnrs}
The following statements hold.
\begin{enumalph}
\item In a clean basis for $T_\ell E$, we have 
\begin{equation}  \label{eqn:EgalK}
\rho_{E,\ell}(\Gal_K) \leq 
\begin{pmatrix}
1+\ell^r\Zl & \ell^{s} \Zl \\
\ell^{r} \Zl & 1+\ell^{s}\Zl
\end{pmatrix}
\end{equation}
where by convention $1+\ell^0\Z_\ell \colonequals \Z_\ell^\times$.  
\item If \eqref{eqn:EgalK} holds in a basis for $T_\ell E$, then $P_{1,r},P_{2,s}$ generate a subgroup of $E[\ell^\infty](K)$ isomorphic to $\Z/\ell^r \times \Z/\ell^s$; if moreover equality holds in \eqref{eqn:EgalK}, then this basis is clean.
\end{enumalph}
\end{lem}

\begin{proof}
Straightforward.
\end{proof}

Now let $n \geq 1$, and for integers $0 \leq r,s \leq n$, define the subgroup
\begin{equation} \label{eqn:Gellrs}
G_\ell(n; r,s) \colonequals \begin{pmatrix}
1+\ell^r\Zl & \ell^{s} \Zl \\
\ell^{n-s} \Zl & 1+\ell^{n-r}\Zl
\end{pmatrix} \leq \GL_2(\Z_\ell)
\end{equation}
with the same convention as in \eqref{eqn:EgalK}.  Indeed, the group on the right-hand side of \eqref{eqn:EgalK} is $G_\ell(r+s;r,s)$, i.e., corresponds to $n=r+s$.  When the prime $\ell$ is clear, we will drop the subscript and abbreviate $G(n;r,s)=G_\ell(n;r,s)$.

Our motivation for studying these groups is indicated by the following lemma.

\begin{lem} \label{lem:hasdenyup}
If $\rho_{E,\ell}(\Gal_K) \leq G(n;r,s)$ for some $0 \leq r,s \leq n$, then $\ell^n \mid \#E(\Fp)$ for all but finitely many $\frakp$.
\end{lem}

\begin{proof}
We see directly that $\det(1-g) \equiv 0 \pmod{\ell^n}$ for all $g \in G(n;r,s)$, so the result follows from \eqref{eqn:detrhoE}.
\end{proof}

\begin{exm} \label{exm:meq2doh}
If $\ell=2$, then $G_2(1;0,0)=G_2(1;1,0)$.
\end{exm}

\begin{exm} \label{exm:rn1od0s}
Suppose that $\rho_{E,\ell}(\Gal_K)=G(n;r,s)$ as above, with $n \geq 1$.

If $r=n$ and $s=0$, then $\overline{\rho}_{E,\ell^n} = \begin{pmatrix} 1 & * \\ 0 & * \end{pmatrix}$ and so if $P_1,P_2 \in E[\ell^n](\Kbar)$ are the $n$th coordinates of the chosen basis for $T_\ell E$, then $E[\ell^\infty](K)=\langle P_1 \rangle \simeq \Z/\ell^n$.  

Similarly, if $r=s=0$ and $\ell \neq 2$, then $\overline{\rho}_{E,\ell^n}=\begin{pmatrix} * & * \\ 0 & 1 \end{pmatrix}$; thus $E[\ell^\infty](K)=\{\infty\}$ and $E$ has a unique cyclic isogeny over $K$ of order $\ell^n$ whose kernel is generated by $P_1$.  

In both cases, we have $\ell^n \mid \#E(\F_\frakp)$ for all but finitely many $\frakp$.  
\end{exm}

Interchanging the basis elements made in the identification \eqref{eqn:tateyup} gives an isomorphism 
\begin{equation} \label{eqn:swapbasis}
G(n;r,s) \xrightarrow{\sim} G(n;n-r,n-s)
\end{equation}
so without loss of generality we may suppose that $r+s \leq n$ (and still that $0 \leq r,s \leq n$).  If $n=r+s$, we have $G(n;r,n-r) \simeq G(n;n-r,r)$.  

\begin{lem} \label{lem:quickpropr}
The following statements hold.
\begin{enumalph}
\item The group $G_\ell(n;r,s)$ is equal to the preimage of its reduction modulo $\ell^{\max(r,s,n-s,n-r)}$.
\item We have $\det G_\ell(n;r,s)=1+\ell^{\min(r,n-r)} \Z_\ell$.  
\item We have 
\[ [\GL_2(\Z_\ell) : G_\ell(n;r,s)] = \begin{cases}
\ell^{2n-3}(\ell^2-1)(\ell-1), & \text{ if $\min(r,n-r) \geq 1$;} \\
\ell^{2n-2}(\ell^2-1), & \text{ if $\min(r,n-r)=0$.}
\end{cases} \]
\item If $\ell^n \geq 5$, then $G_\ell(n;r,s) \cap \SL_2(\Z)$ is torsion free.  
\end{enumalph}
\end{lem}

\begin{proof}
Parts (a) and (b) follow from a direct calculation.  For part (c), we reduce modulo $n$ (using (a)) and count the size of the reduction in each coordinate: we find
\[ \frac{\phi(\ell^n)}{\phi(\ell^r)} \ell^{n-s} \ell^s \frac{\phi(\ell^n)}{\phi(\ell^{n-r})} = \ell^{3n-2}(\ell-1)^2 \cdot \begin{cases} 
(\ell^{n-2}(\ell-1)^2)^{-1}, & \text{ if $r,n-r \geq 1$;} \\
(\ell^{n-1}(\ell-1))^{-1}, & \text{ if $r=0,n$.} \end{cases} \] 
Simplifying and noting $\#\GL_2(\Z/\ell^n)=\ell^{4(n-1)}\#\GL_2(\Z/\ell)=\ell^{4n-3}(\ell-1)(\ell^2-1)$, the result follows.

For part (d), as in Lemma \ref{lem:hasdenyup} we have $\det(g-1) \equiv 0 \pmod{\ell^n}$ for all $g \in G(n;r,s)$.  If $g \in \SL_2(\Z)$ is torsion, then its characteristic polynomial matches that of a root of unity of order dividing $6$, and if $g \neq 1$ then $\det(g-1)=1, 2, 3, 4$, a contradiction.  
\end{proof}

The groups $G(n;r,s)$ arise from curves isogenous to the ones studied in Lemma \ref{lem:cleantelnrs}, as follows.

\begin{prop} \label{prop:cleanCvrall}
Suppose in a (clean) basis for $T_\ell E$ that we have 
\[ \rho_{E,\ell}(\Gal_K)=G(n;r,n-r). \]  
Then the following statements hold.
\begin{enumalph}
\item Any cyclic subgroup $C \leq E[\ell^\infty](\Kbar)$ stable under $\Gal_K$ in fact has $C \leq E[\ell^\infty](K)$.
\item If $\varphi \colon E \to E'$ is a cyclic isogeny with $\deg \varphi=\ell^k$, then $k \leq \max(r,n-r)$.  Moreover, there exists a clean basis for $E'$ such that 
\begin{equation} \label{eqn:whatdowemean}
\rho_{E',\ell}(\Gal_K)=
\begin{cases}
G(n;r,n-r-k), & \text{only if $k \leq n-r$;}  \\
G(n;n-r,r-k), & \text{only if $k \leq r$.}
\end{cases} 
\end{equation}
\end{enumalph}
\end{prop}

In \eqref{eqn:whatdowemean}, we mean that if $r < k \leq n-r$ then the first case must occur, and symmetrically if $n-r < k \leq r$ then the second case must occur; if $k \leq \min(r,n-r)$, then either case can arise.  The proof will show that all possibilities do arise.

\begin{proof}
We first prove (a).  Interchanging basis elements as in \eqref{eqn:swapbasis}, we may suppose without loss of generality that $r \leq n-r$.  Let $\#C=\ell^k$.  A generator for $C$ is of the form $P = x_1P_{1,k} + x_2P_{2,k}$ with $(x_1:x_2) \in \PP^1(\Z/\ell^k)$. 
\begin{itemize}
\item If $k \leq r$, then we have full $\ell^k$-torsion $E[\ell^k](\Kbar)=E[\ell^k](K)$ and so certainly (a) holds.  
\item Suppose $r< k \leq n-r$.  Then by hypothesis, 
\[ \overline{\rho}_{E,\ell^k}(\Gal_K) = \left\{ \begin{pmatrix} 1 + \ell^r a & 0 \\ \ell^r c & 1 \end{pmatrix} : a,c \in \Z/\ell^{k-r} \right\}. \]
Taking $a \in \Z/\ell^{k-r}$, we see that the only stable lines (eigenvectors) in the action on column vectors are generated by elements $(x:1) \in \PP^1(\Z/\ell^k)$ with $\ell^{k-r} \mid x$, so $P=xP_{1,k} + P_{2,k} = (x/\ell^{k-r})P_{1,r} + P_{2,k} \in E[\ell^k](K)$ as desired.
\item Finally, if $k > s \colonequals n-r$, then $\ell^{k-s} C$ is Galois stable, so $(x_1:x_2) \equiv (0:1) \pmod{\ell^s}$ by the previous case.  Applying upper-triangular unipotent matrices in $\overline{\rho}_{E,\ell^k}(\Gal_K)$ then gives a contradiction.
\end{itemize}

Next, part (b).  Lemma \ref{lem:cleantelnrs}(b) and the preceding part (a) imply $k \leq \max(r,n-r)$.  We now change convention for convenience, supposing that $n-r \leq r$.  Let $P_1,P_2$ be the given clean basis for $T_\ell E$.  As in (a), let $C=\ker \varphi$ be generated by $P=x_1 P_{1,k} + x_2 P_{2,k}$ with $(x_1:x_2) \in \PP^1(\Z/\ell^k)$.  We consider two cases.  

First, suppose $x_1=0$.  Then $(x_1:x_2)=(0:1)$ and $C$ is generated by $P_{2,k}$.  A basis for $T_\ell E'$ (in $V_\ell E$) is then given by $P_1, \ell^{-k} P_2$.  In this basis and pulling the scalar $\ell^k$ through, we have
\begin{equation} \label{eqn:rhoEell}
\rho_{E',\ell}(\Gal_K) =\begin{pmatrix}  \ell^{k} & 0 \\ 0 & 1 \end{pmatrix} G(n;r,n-r) \begin{pmatrix}  \ell^{-k} & 0 \\ 0 & 1 \end{pmatrix}  = G(n;r,n-r+k)
\end{equation}
and $k \leq n-r$ (with still $n-r \leq r$).  Applying \eqref{eqn:swapbasis} converts $G(n;r,n-r+k)$ to $G(n;n-r,r-k)$, as claimed.  

Otherwise, we have $x_1 \neq 0$; then $(x_1:x_2)=(1:x)$ for $x \in \Z/\ell^k$; we lift to $x \in \Z_\ell^\times$.  Let
\[ U \colonequals \begin{pmatrix} 0 & 1 \\ 1 & x \end{pmatrix} \in \GL_2(\Z_\ell). \]
Let $Q_i=UP_i$ for $i=1,2$, so $Q_1,Q_2$ is a (no longer necessarily clean) basis for $T_\ell E$ in which $\ker \varphi = \langle Q_{2,k} \rangle$.  Nevertheless, we calculate:
\begin{equation}
\begin{pmatrix} \ell^k & 0 \\ 0 & 1 \end{pmatrix} U = \begin{pmatrix} \ell^k & 0 \\ 0 & 1 \end{pmatrix} \begin{pmatrix} 0 & 1 \\ 1 & x \end{pmatrix} = \begin{pmatrix} 0 & \ell^k \\ 1 & x \end{pmatrix} = \begin{pmatrix} 0 & 1 \\ 1 & 0 \end{pmatrix} \begin{pmatrix} 1 & x \\ 0 & \ell^k \end{pmatrix} = U'A.
\end{equation}
A basis for $T_\ell E'$ (in $V_\ell E)$ is given by $Q_1,\ell^{-k}Q_2$, so the same is true after applying $(U')^{-1}=U'$ which just swaps basis vectors to give $\ell^{-k}Q_2,Q_1$.  In this basis, the image of $\rho_{E',\ell}(\Gal_K)$ is
\[ U'\begin{pmatrix} \ell^k & 0 \\ 0 & 1 \end{pmatrix} U G(n;r,n-r) U^{-1} 
\begin{pmatrix} \ell^{-k} & 0 \\ 0 & 1 \end{pmatrix} U' = A G(n;r,n-r) A^{-1} \]
and we then compute:
\begin{equation} 
\begin{aligned}
&\begin{pmatrix} 1 & x \\ 0 & \ell^k \end{pmatrix} \begin{pmatrix} 1 + \ell^r a & \ell^{n-r} b \\ \ell^r c & 1 + \ell^{n-r} d \end{pmatrix} \begin{pmatrix} 1 & -x\ell^{-k} \\ 0 & \ell^{-k} \end{pmatrix} \\
&= \begin{pmatrix} 1 + \ell^r(a+cx) & -(a+cx)(x\ell^{r-k}) + (b+dx)\ell^{n-r-k}   \\ 
c\ell^{k+r} & 1+d\ell^{n-r}-cx\ell^r \end{pmatrix}.
\end{aligned}
\end{equation}
We are free to reparametrize, replacing $a,b \leftarrow a+cx,b+dx$ to get
\begin{equation} 
= \begin{pmatrix} 1 + \ell^r a & b\ell^{n-r-k}-ax\ell^{r-k}    \\ 
c\ell^{k+r} & 1+d\ell^{n-r}-cx\ell^r \end{pmatrix}.
\end{equation}
Since $n-r \leq r$, then we recognize the group $G(n;r,n-r-k)$.  Putting these together gives the result.
\end{proof}

Recall that the $\ell$-\defi{isogeny graph} of $E$ has as vertices the set of curves $\ell$-power isogenous to $E$ up to isomorphism and (undirected) edges are $\ell$-isogenies.  Proposition \ref{prop:cleanCvrall} provides a description of the $\ell$-isogeny graph of $E$ when $\rho_{E,\ell}(\Gal_K)=G(n;r,n-r)$ (depending essentially only on $n$)---a nontrivial path in the graph is a cyclic $\ell$-power isogeny.  Here are a few illustrative examples.

\begin{exm} \label{exm:isoggraph}
Suppose $\rho_{E,\ell}(\Gal_K)=G(n;r,n-r)$ with $n \geq 1$, and without loss of generality suppose $r \leq n-r$.  

If $r=0$, then $E$ has Galois image $G(n;0,n)$, and the isogeny graph consists of a chain of $n+1$ vertices with Galois images $G(n;0,n),G(n;0,n-1),\dots,G(n;0,0)$; the kernels of these isogenies are cyclic subgroups of $E[\ell^\infty](K) \simeq \Z/\ell^n$.  

For Galois image $G(2;1,1)$ ($n=2$ and $r=1$), there are $\ell+1$ vertices adjacent to $E$ with Galois image $G(2;1,0)$.  
\end{exm}

We conclude this section by a study of curves with Galois image (contained in) $G(n;r,s)$, building on Lemma \ref{lem:cleantelnrs}.

\begin{lem} \label{lem:charimage}
Suppose $\rho_{E,\ell}(\Gal_K) \leq G(n;r,s)$ with $0 \leq r,s,r+s \leq n$, and if $\ell=2$ suppose that $(r,s) \neq (0,0)$.  Then the following statements hold.
\begin{enumalph}
\item \label{charimagepart1} 
If $\rho_{E,\ell}(\Gal_K)=G(n;r,s)$, then $P_{1,r},P_{2,s}$ generate $E[\ell^\infty](K)$; in particular, we have $E[\ell^\infty](K) \simeq \Z/\ell^r \times \Z/\ell^s$.  
\item Suppose $\rho_{E,\ell}(\Gal_K) \leq G(n;r,s)$.  Then, for all $t$ such that $s \leq t \leq n-r$, there exists a cyclic $\ell^{t-s}$ isogeny $E \to E'$ over $K$ such that if $P_1,P_2$ is a basis for $T_\ell E$, then
$$
\rho_{E',\ell}(\Gal_K) \leq G(n;r,t)
$$
in the basis $\ell^{s-t}P_1, P_2$ for $T_\ell E'$ (in $V_\ell E$).
\item The elliptic curve $E$ admits a cyclic $\ell^{n-(r+s)}$-isogeny $E \to E'$ over $K$ with $\ell^n \mid \#E'(K)\sbtors$.
\end{enumalph}
\end{lem}

\begin{proof}
We prove (a), and let $P_{1,n},P_{2,n} \in E[\ell^n](\Kbar)$ be the $n$th coordinates of the chosen basis for $T_\ell E$, and consider a point $P \colonequals x_1 P_{1,n} + x_2 P_{2,n} \in E[\ell^n](\Kbar)$ with $x_1,x_2 \in \Z/\ell^n$.  Of course $P \in E[\ell^n](K)$ if and only if $(g-1)(P) \equiv 0 \pmod{\ell^n}$ for all $g \in G(n;r,s)$.  If $P \in E[\ell^n](K)$, then taking diagonal matrices shows that $\ell^{n-r} \mid x_1$ and $\ell^{r} \mid x_2$; since $r+s \leq n$, we have $s \leq n-r$ so $\ell^{s} \mid x_1$ and similarly $\ell^{n-s} \mid x_2$.  Conversely, 
\begin{equation}
\begin{pmatrix}
\ell^r a & \ell^s b \\
\ell^{n-s} c & \ell^{n-r} d 
\end{pmatrix} 
\begin{pmatrix} \ell^{n-r} \\ \ell^{n-s} \end{pmatrix} 
\equiv 0 \pmod{\ell^n}
\end{equation}
so $E[\ell^n](K)=\langle \ell^{n-r} P_1, \ell^{n-s} P_2 \rangle \simeq \Z/\ell^r \times \Z/\ell^s$, proving (a).

Next, part (b).  Let $u \in \Z$ satisfy $s \leq u \leq n$.  A similar argument in coordinates as in the previous paragraph shows $\ell^u P_{1,n}$ generates a Galois-stable subgroup of $E(K)$.  Let $E' \colonequals E/\langle \ell^u P_{1,n} \rangle$, so that the quotient map $E \to E'$ defines a cyclic $\ell^{n-u}$-isogeny.  Conjugating as in \eqref{eqn:rhoEell} shows that $\rho_{E',\ell}(\Gal_K) = G(n;r,s+n-u)$.  Restricting $u$ to range over $r+s \leq u \leq n$, the image $\rho_{E',\ell}(\Gal_K)$ ranges over $G(n;r,t)$ with for $s \leq t \leq n-r$, with $n-u=t-s$.  

Finally, for (c), take $t=n-r$ in part (b).
\end{proof}

\subsection{Refining the theorem of Katz} \label{sec:refinement}

In this section, we refine the result of Katz (mentioned in the introduction), which we now recall.  

\begin{thm}[{Katz \cite{katz}}] \label{thm:katz}
Let $n \geq 1$.  Suppose that $\ell^n \mid \#E(\Fp)$ for a set of good primes of $K$ of density $1$.  Then there exists an elliptic curve $E'$ over $K$ that is $K$-isogenous to $E$ and a $\Z_\ell$-basis of $T_\ell E' \simeq \Z_\ell^2$ such that
\begin{equation} \label{eqn:rhoEpell}
\rho_{E',\ell}(\Gal_K) \leq G(n;r,n-r) 
\end{equation}
for some integer $0 \leq r \leq n$.  In particular, $\ell^n \mid \#E'(K)\sbtors$.
\end{thm}

\begin{proof}
We briefly review the method of proof for the reader's convenience.  (Some details of the argument are explained in the next section.)  Let $V$ be a 2-dimensional $\Ql$-vector space and let $G \leq \Aut(V)$ be a compact open subgroup.  By an inductive group-theoretic argument, Katz \cite[Theorem 1]{katz} shows that if 
$$
\det(1-g) \equiv 0 \pmod{\ell^n}
$$
holds for all $g \in G$, then there exist $G$-stable lattices $\mathcal{L}' \subseteq\mathcal{L} \subseteq V$ such that the quotient $\mathcal{L}/\mathcal{L}'$ has order $\ell^n$ and trivial $G$-action; equivalently, there exists a $\Ql$-basis of $V$ such that $G \leq G(n;r,n-r)$ for some integer $0 \leq r \leq n$.  

We then apply the preceding paragraph to elliptic curves \cite[Theorem 2]{katz}.  We take $V = T_\ell E \otimes {\Q_\ell}$ and $G=\Gal_K$; then $\mathcal{L}'=T_{\ell}(E')$ for some elliptic curve $E'$ over $K$ that is $K$-isogenous to $E$, and 
\begin{equation} 
\mathcal{L}/\mathcal{L}' \subseteq \ell^{-n}\mathcal{L'}/\mathcal{L'} \simeq \mathcal{L}'/\ell^n \mathcal{L}' \simeq E'[\ell^n] 
\end{equation} is a subgroup of $K$-rational torsion points of $E'$ (see \cite[Introduction]{katz} for a review of Galois-stable lattices of the Tate module).
\end{proof}

\begin{lem} \label{lem:phicycl}
Under the hypotheses of Theorem \textup{\ref{thm:katz}}, the isogeny $\varphi\colon E \to E'$ may be taken to be a cyclic $\ell$-power isogeny.
\end{lem}

\begin{proof}
Given any isogeny $\varphi\colon E \to E'$, we may factor $\varphi$ into first an isogeny of $\ell$-power degree then an isogeny of degree coprime to $\ell$.  The latter isogeny preserves the image of $\rho_{E',\ell}$, so we may assume $\varphi$ has $\ell$-power degree.  The resulting isogeny factors as a cyclic $\ell$-power isogeny followed by multiplication by a power of $\ell$, and again the latter preserves the image of $\rho_{E',\ell}$, so the conclusion follows.
\end{proof}

To refine Theorem \ref{thm:katz}, we identify the image of $\rho_{E,\ell}$ by following the isogeny guaranteed by Lemma \ref{lem:phicycl}.  In general, one can say little more than $E$ is isogenous to $E'$!  The following lemma is the starting point for Katz, as it is for us.

\begin{lem} \label{serre_exercise}
Let $k$ be a field, let $V$ be a $k$-vector space with $\dim_k V=2$, and let $G \leq \GL(V)$ be a subgroup.  Suppose that $\det(1-g)=0$ for all $g \in G$.  Then there exists a basis of $V \simeq k^2$ such that $G \leq \GL_2(k)$ is contained one of the subgroups
\[ \begin{pmatrix} 1 & k \\ 0 & k^\times \end{pmatrix} \quad \text{or} \quad \begin{pmatrix} k^\times & k \\ 0 & 1 \end{pmatrix}. \]
\end{lem}

\begin{proof}
See Serre \cite[p.~I-2, Exercise 1]{serre}; when $k$ is perfect, see the proof by Katz \cite[Lemma 1, p.~484]{katz} using the Brauer--Nesbitt theorem.  
\end{proof}

\begin{cor} \label{cor:EKell}
If $\ell \mid \#E(\Fp)$ for a set of primes of $K$ of density $1$, then at least one of the following holds:
\begin{enumroman}
\item $E(K)[\ell] \neq \{\infty\}$; or
\item there is a cyclic $\ell$-isogeny $E \to E'$ over $K$ where $E'(K)[\ell] \neq \{\infty\}$.
\end{enumroman}
\end{cor}

\begin{proof}
Apply Lemma \ref{serre_exercise} with $k=\F_\ell$ and $V=E[\ell]$, and $G=\overline{\rho}_{E,\ell}(\Gal_K)$.  
For the first subgroup we are in case (i); for the second, the basis $P_1,P_2$ provided by the lemma gives an $\ell$-isogenous curve $E' \colonequals E/\langle P_1 \rangle$ over $K$ with the image of $\langle P_2 \rangle$ invariant under $G$, so we are in case (ii).
\end{proof}

In other words, Corollary \ref{cor:EKell} says that when $n=1$, we may take the isogeny $\varphi\colon E \to E'$ provided by Lemma \ref{lem:phicycl} to have degree dividing $\ell$; in particular, this proves a refinement of Theorem \ref{thm:katz} for $n=1$.

We now seek to generalize Corollary \ref{cor:EKell} to the prime power case $m=\ell^n$.  We start by considering the case where the degree of the isogeny $\varphi\colon E \to E'$ provided by Lemma \ref{lem:phicycl} is large.

\begin{lem} \label{mubound}
Let $\varphi \colon E \to E'$ be a cyclic $\ell^k$-isogeny over $K$ such that $\ell^n \mid \#E'(K)\sbtors$.  Suppose that $k \geq n$.  Then there is a $\Z_\ell$-basis for $T_\ell E \simeq \Z_\ell^2$ such that
\begin{equation} \label{eqn:rhoEellnr0}
\rho_{E,\ell}(\Gal_K) \leq G(n;r,0) = \begin{pmatrix} 1 + \ell^r \Zl & \Zl \\ \ell^{n} \Zl & 1+ \ell^{n-r} \Zl\end{pmatrix}
\end{equation}
for some integer $0 \leq r \leq n$.  In particular, there exists a cyclic $\ell^{n-r}$-isogeny $\psi\colon E \to E''$ such that $\ell^n \mid \#E''(K)\sbtors$ and $\rho_{E'',\ell}(\Gal_K) \leq G(n;r,n-r)$. 
\end{lem}

\begin{proof}
By hypothesis, there is a cyclic subgroup $C_k \leq E(\Kbar)$ stable under $\Gal_K$ of order $\ell^k$.  Since $k \geq n$, the subgroup $\ell^{k-n}C_k \leq E(\Kbar)$ is also $\Gal_K$-stable and order $\ell^n$.  Extending to a basis for $E[\ell^n](\Kbar)$, we have
\begin{equation} \label{eqn:GrhoellnK}
G \colonequals \overline{\rho}_{E,\ell^n}(\Gal_K) \leq \begin{pmatrix} * & * \\ 0 & * \end{pmatrix}.
\end{equation}
The containment \eqref{eqn:rhoEellnr0} is determined by reduction modulo $\ell^n$, so equivalently we show 
\begin{equation} \label{eqn:1rzldns}
G \leq \begin{pmatrix} 1+\ell^r\Z/\ell^n & * \\ 0 & 1+\ell^{n-r}\Z/\ell^n \end{pmatrix} 
\end{equation}
for some $r$.

Since $\ell^n \mid \#E'(K)\sbtors$, as in \eqref{eqn:detrhoE} we conclude that $\det(1-g) \equiv 0 \psmod{\ell^n}$ for all $g \in G$.  Let $g=\begin{pmatrix} a & b \\ 0 & d \end{pmatrix} \in G$ be such that $r \colonequals \ord_\ell(1-a)$ minimal, so that $0 \leq r \leq n$.
Then
\begin{equation} \label{eqn:detcond}
\det(1-g)=(1-a)(1-d) \equiv 0 \pmod{\ell^n} 
\end{equation}
gives $d \equiv 1 \psmod{\ell^{n-r}}$, which is a start.  To finish, let $g'=\begin{pmatrix} a' & b' \\ 0 & d' \end{pmatrix} \in G$ be any element, and let $r' \colonequals \ord_\ell(1-a) \geq r$.  Then $\ord_\ell(1-d') \geq n-r'$ as in \eqref{eqn:detcond}, so if $r'=r$ we are done.  So suppose $r'>r$.  Consider the determinant condition on $gg'$, which reads
\begin{equation}
\det(1-gg') = (1-aa')(1-dd') \equiv 0 \pmod{\ell^n}. 
\end{equation}
Then $aa' \equiv a \not\equiv 1 \pmod{\ell^{r+1}}$, so $\ord_\ell(1-aa')=r$, and thus $\ord_\ell(1-dd') \geq n-r$, i.e., $dd' \equiv 1 \psmod{\ell^{n-r}}$.  But we already have $d \equiv 1 \psmod{\ell^{n-r}}$, so $d' \equiv 1 \psmod{\ell^{n-r}}$, proving \eqref{eqn:1rzldns}.  

The final statement then follows from Lemma \ref{lem:charimage}(b), with $s=0$.
\end{proof}

\begin{cor} \label{cor:sumitup}
For $m \geq 1$, suppose that $m \mid \#E(\F_\frakp)$ for a set of primes of $K$ of density $1$.  Then there exists a cyclic isogeny $\varphi \colon E \to E'$ of degree $d \mid m$ such that $m \mid \#E'(K)\sbtors$.  Moreover, for every $\ell^n \parallel m$, there exists $0 \leq r \leq n$ (depending on $\ell$) such that 
\[ \rho_{E',\ell}(\Gal_K) \leq G_\ell(n;r,n-r). \]
\end{cor}

\begin{proof}
For each prime power $\ell^n \parallel m$, apply the theorem of Katz (Theorem \ref{thm:katz}), the refinements of Lemmas \ref{lem:phicycl} and \ref{mubound}; and then combine these isogenies (taking the sum of the kernels). 
\end{proof}

By Corollary \ref{cor:sumitup}, the possible elliptic curves $E$ that locally have a subgroup of order $m=\ell^n$ arise (dually) from cyclic isogenies from curves with $\ell$-adic Galois images contained in $G_\ell(n;r,n-r)$ for some $r$.  To conclude, we add the hypothesis that this latter containment is an \emph{equality}; when we calculate probabilities, we will see this fullness condition holds outside of a negligible set.

\begin{thm} \label{upgrade}
For $m \geq 1$, suppose that $m \mid \#E(\F_\frakp)$ for a set of primes of $K$ of density $1$, and let $\varphi\colon E \to E'$ be a cyclic $\ell$-power isogeny over $K$ such that $\rho_{E',\ell}(\Gal_K) = G(n;r,n-r)$ (in a choice of basis for $T_\ell E'$) for some $0 \leq r \leq n$.  Then there exists $s$ with $0 \leq s \leq n$ such that $\rho_{E,\ell}(\Gal_K) = G(n;r,s)$ (in a basis for $T_\ell E$).
\end{thm}

\begin{proof}
We have $\deg \varphi=\ell^k$ for some $k \geq 0$.  We apply Proposition \ref{prop:cleanCvrall}(b) to the dual isogeny $\varphi\spcheck \colon E' \to E$ (with, alas, the roles of $E$ and $E'$ interchanged): we conclude that $\rho_{E,\ell}(\Gal_K) = G(n;r,s)$ with $0 \leq s \leq n-r \leq n$ or $\rho_{E,\ell}(\Gal_K) = G(n;n-r,s')$ with $s' \leq r \leq n$.  In the latter case, recalling \eqref{eqn:swapbasis}, we have equivalently $\rho_{E,\ell}(\Gal_K) = G(n;r,s)$ with $0 \leq s=n-s' \leq n$.
\end{proof}

\section{Counting elliptic curves} \label{sec:moregen}

In this section, we count by height elliptic curves parametrized by a modular curve of genus zero uniformized by a torsion free congruence subgroup.

\subsection{Moduli of elliptic curves} \label{sec:modellcur}

We quickly set up the necessary theory concerning moduli of elliptic curves.  

Let $G \leq \GL_2(\Z/N)$ be a subgroup.  If $G$ arises as the image of the mod $N$ Galois representation of an elliptic curve over $\Q$, then its determinant is the cyclotomic character and thus surjective, so we suppose that $\det G=(\Z/N)^\times$.  Let $\pi_N \colon \SL_2(\Z) \to \SL_2(\Z/N)$ be the projection and as in \eqref{eqn:GammaG} let 
\[ \Gamma_G \colonequals \pi_N^{-1}(G \cap \SL_2(\Z/N)) \leq \SL_2(\Z). \]
The group $\Gamma_G$ is a discrete group acting properly on the upper half-plane $\Htwo$, and the quotient $\Gamma_G \backslash \Htwo$ can be given the structure of a Riemann surface (compact minus finitely many points).  Attached to $G$ is the moduli problem of elliptic curves with $G$-level structure, as in the following proposition.  

\begin{prop} \label{prop:theresacurve}
Suppose that $\det G=(\Z/N)^\times$.  Then there exists an affine, smooth, geometrically integral curve $Y_G$ defined over $\Q$, unique up to isomorphism, with the following properties.
\begin{enumroman}
\item There is an isomorphism of Riemann surfaces $\Gamma_G \backslash \Htwo \xrightarrow{\sim} Y_G(\C)$.  
\item For every number field $K$, there is a (functorial) bijection between the set $Y_G(K)$ and the set of $\Kbar$-isomorphism classes of $\Gal_K$-stable $G$-equivalence classes of pairs $(E,\iota)$, where $E$ is an elliptic curve over $K$ and $\iota \colon E[N](\Kbar) \to (\Z/N)^2$ is an isomorphism of groups.
\item For every elliptic curve $E$ over $K$, there exists $\iota$ such that the isomorphism class of $(E,\iota)$ lies in $Y_G(K)$ if $\overline{\rho}_{E,N}(\Gal_K) \lesssim G$ is contained in a subgroup conjugate to $G$; the converse holds if $j(E) \neq 0,1728$.
\item If $\Gamma_G$ is torsion free (in particular $-1 \not\in G$), then \textup{(ii)} holds but for $K$-isomorphism classes, and there is a universal elliptic curve $E_{G,\textup{univ}} \to Y_G$, unique up to isomorphism.
\end{enumroman}
\end{prop}

In (iv), in particular, $E_{G,\textup{univ}}$ is an elliptic curve over (the affine coordinate ring of) $Y_G$, and the bijection in (iv) is defined by the map that sends $P \in Y_G(K)$ to the fiber of $E_{G,\textup{univ}} \to Y_G$ over $P$.

\begin{proof}
The curve $Y_G$ can be constructed as the quotient of the (connected but geometrically disconnected) modular curve $Y(N)$ defined over $\Q$ by $G$.  For more details, see Deligne--Rapoport \cite[Chapters IV, VI]{DR} or the tome of Katz--Mazur \cite[Chapter 4]{KM}; for property (iii), see Baran \cite[\S 4]{Baran} and Zywina \cite[Proposition 3.2]{zywina}.
\end{proof}

We recall also here the notion of an \emph{irregular cusp} (see e.g., Diamond--Shurman \cite[(3.3), p.~75]{DS}, Shimura \cite[\S 2.1, p.~29]{Shimura}), primarily to show it is only a minor nuisance.  Let $\Gamma \leq \SL_2(\Z)$ be a subgroup of finite index.  If $-1 \in \Gamma$, then every cusp of $\Gamma$ is regular; so suppose $-1 \not\in \Gamma$.  Then the stabilizer of the cusp $\infty$ under $\Gamma$ is an infinite cyclic group generated by $\pm \begin{pmatrix} 1 & h \\ 0 & 1 \end{pmatrix}$ for some $h \in \Z_{>0}$, and we accordingly say that $\infty$ is \defi{regular} or \defi{irregular} as the sign of this generator is $+$ or $-$.  For any cusp $s$, we choose a matrix $\alpha \in \SL_2(\Z)$ such that $\alpha(\infty)=s$ and conjugate the preceding definition.  

The groups $G_\ell(n;r,s) \leq \GL_2(\Zl)$ of Section \ref{S1.6} naturally define subgroups $\overline{G}_{\ell^n}(n;r,s) \leq \GL_2(\Z/\ell^n)$ by reduction modulo $\ell^n$.

\begin{lem} \label{lem:noirreg}
Let $\ell$ be prime, let $n \geq 1$, and for integers $0 \leq r,s \leq n$ with $r+s \leq n$, let $G=\overline{G}_{\ell^n}(n;r,s) \leq \GL_2(\Z/\ell^n)$ be the reduction modulo $\ell^n$ of $G_\ell(n;r,s)$.  Then the group $\Gamma_G$ has no irregular cusps except when $\ell^n=2^2=4$ and $rs=0$.
\end{lem}

\begin{proof}
If $\gamma \in \Gamma_G$, then $\gamma=\begin{pmatrix} 1+\ell^r a_0 & \ell^s b_0 \\ \ell^{n-s} c_0 & 1+\ell^{n-r}d_0 \end{pmatrix}$ with $a_0,b_0,c_0,d_0 \in \Z$ and 
\begin{equation} \label{eqn:r10d8sa}
\begin{aligned}
\det(\gamma)&=(1+\ell^r a_0)(1+\ell^{n-r} d_0)-\ell^n b_0c_0 \\
&= 1 + \ell^r a_0 + \ell^{n-r} d_0 + \ell^n(a_0d_0-b_0c_0) = 1, 
\end{aligned}
\end{equation}
so expanding we find 
\begin{equation} \label{tryup}
\tr(\gamma)= 2+ \ell^r a_0+\ell^{n-r}d_0 \equiv  2 \pmod{\ell^n}.
\end{equation}

Let $s$ be a cusp of $\Gamma_G$ and $\alpha \in \SL_2(\Z)$ be such that $\alpha(\infty)=s$, and consider the group $\alpha^{-1} \Gamma_G \alpha$.  Let $\alpha^{-1}\gamma\alpha = \pm \begin{pmatrix} 1 & h \\ 0 & 1 \end{pmatrix} \in \alpha^{-1} \Gamma_G \alpha$ generate the stabilizer of $\infty$.  Then $\tr(\alpha^{-1}\gamma\alpha)=\tr(\gamma) = \pm 2 \equiv 2 \pmod{\ell^n}$.  Suppose $s$ is irregular.  Then $-2 \equiv 2 \pmod{\ell^n}$ so $\ell^n=2^1,2^2$.  If $\ell^n=2^1$ then $-1 \in \Gamma_G$ and $s$ is regular by definition.  So suppose $\ell^n=2^2$.  We have the cases $(r,s)=(0,0),(1,0),(1,1),(2,0)$.  If $r=1$ then again $-1 \in \Gamma_G$.  Otherwise, $(r,s)=(2,0),(0,0)$ then by Example  \ref{exm:rn1od0s} we see $\Gamma_G=\Gamma_1(4)$, and $1/2$ is indeed an irregular cusp \cite[Exercise 3.8.7]{DS}.  
\end{proof}

\begin{lem} \label{lem:samecurve}
We have 
\[ \Gamma_{\overline{G}_{\ell^n}(n;r,s)} = \Gamma_{\overline{G}_{\ell^n}(n;r',s)} \]
where $r' \colonequals \max(r,n-r)$.
\end{lem}

\begin{proof}
Looking back at \eqref{eqn:r10d8sa}, we see that e.g.\ if $r \leq n-r$ then $1+\ell^ra_0 \equiv 1 \pmod{\ell^{n-r}}$, so $\ell^{n-r} \mid \ell^{r}a_0$.  
\end{proof}

Lemma \ref{lem:samecurve} indicates one of the ways in which different moduli problems can have the same underlying uniformizing congruence subgroup.

To complete our setup for our main result (Theorem \ref{thm:hsupgrade0}), we must decide how to count our elliptic curves.  Specifically, we need to distinguish between counting elliptic curves $E$ for which \emph{there exists} a rational $G$-structure, versus counting equivalence classes of pairs $(E,\iota)$ of elliptic curves $E$ equipped with rational $G$-structures $\iota$.  Ultimately, we will see that these two counts differ by a simple multiple (on the main term, and  with square root error term).  

To this end, for an elliptic curve $E$ over $\Q$, let $r_G(E)$ be the number of $\Kbar$-isomorphism classes of $\Gal_K$-stable $G$-equivalence classes of pairs $(E,\iota)$ as in Proposition \ref{prop:theresacurve}(b), equivalently the number of isomorphism classes $[(E,\iota)] \in Y_G(\Q)$.  Let 
\begin{equation} \label{eqn:rG}
r(G) \colonequals [N_{\GL_2(\Z/N)}(G) : G] 
\end{equation}
be the index of $G$ in its normalizer in $\GL_2(\Z/N)$.  Write 
\[ \pm G \colonequals G\langle -1 \rangle = G \cup -G, \] 
so $\pm G = G$ if and only if $-1 \in G$.  Then $N_{\GL_2(\Z/N)}(\pm G)=N_{\GL_2(\Z/N)}(G)$, so $r(G)=2r(\pm G)$ if $-1 \not \in G$.

%To be precise, let $r(G)$ be the number of pairs $(E,\iota)$, where $\iota$ is a strict $G$-structure (\emph{i.e.}~gives $G$ structure, but not a $G'$-structure for any proper subgroup $G' \leq G$).   Concretely, if $W$ is the subscheme of $\PP^3_\Z$ defined by 
%\begin{align*}
%A(a,b)  &= A(a',b') \\
%B(a,b) &= B(a',b'), 
%\end{align*}
%then $r(G)$ is the degree of the map $\PP^3_\Z \to \PP^1_\Z$ given by $(a:b:a':b') \mapsto (a:b)$.

\begin{exm} \label{exm:GZN3}
If $G=\begin{pmatrix} 1 & * \\ 0 & * \end{pmatrix}$, then $N_{\GL_2(\Z/N)}(G) = \begin{pmatrix} * & * \\ 0 & * \end{pmatrix}$ and so $r(G) = \phi(N)=[\Gamma_0(N):\Gamma_1(N)]$ and $r(\pm G)=\phi(N)/2$ when $N \geq 3$.  
\end{exm}

\begin{lem} \label{g_structure_lem}
Let $E$ be an elliptic curve over $\Q$ with $j(E) \neq \{0,1728\}$.  Then the following statements hold.
\begin{enumalph}
\item If $r_G(E) \geq 1$, then $r_G(E) \geq r(\pm G)$.
\item If $r_G(E) > r(\pm G)$, then there exists a proper subgroup $G'<G$ such that $r_{G'}(E) \geq 1$.  
\end{enumalph}
\end{lem}

\begin{proof}
First, part (a).  By the description in Proposition \ref{prop:theresacurve}(ii), the group $N_{\GL_2(\Z/N)}(G)$ acts functorially on moduli points (postcomposing after $\iota$), so it acts by automorphisms of $Y_G$ defined over $\Q$.  In particular, this group acts on the set of isomorphism classes $[(E,\iota)] \in Y_G(\Q)$ that counted by $r_G(E)$.  We claim that the stabilizer of this action is $\pm G$.  Indeed, let $u \in N_{\GL_2(\Z/N)}(G)$ and suppose that $[(E,\iota)]=[(E,u\iota)]$.  Then there exists an automorphism $\alpha \in \Aut(E)$ such that $G\iota = G u\iota\alpha$.  Since $j(E) \neq \{0,1728\}$ we have $\Aut(E)=\{\pm 1\}$, hence $G\iota = G u \iota \alpha=Gu\alpha\iota$ so $\pm G u=\pm G$, i.e., $u \in \pm G$.  This proves (a).

We now prove (b).  In view of Proposition \ref{prop:theresacurve}(iii), we may prove the contrapositive: if the image $\rho_{E,N}(\Gal_\Q)\lesssim G \leq \GL_2(\Z/N)$ is \emph{onto} $G$ (up to conjugacy), then in fact $r_G(E)=r(\pm G)$.  Indeed, let $[(E,\iota)],[(E,\iota')] \in Y_G(\Q)$.  
% In concrete terms, these correspond to distinct pairs $(a,b)$, $(a',b')$ in the description immediately preceding this Lemma.  
Then the isomorphisms $\iota,\iota' \colon E(\Kbar)[N] \to (\Z/N)^2$ may be chosen such that the two representations $\rho_{E,N},\rho_{E,N}' \colon \Gal_\Q \to \GL_2(\Z/N)$ are subgroups of $G$.  Let $u \colonequals \iota' \iota^{-1} \in \GL_2(\Z/N)$; then the matrix $u$ conjugates the image of $\rho_{E,N}$ into $\rho_{E,N}'$.  But since $\rho_{E,N}(\Gal_\Q)=G$ by hypothesis and $\rho_{E,N}'(\Gal_\Q) \leq G$, we must have $u \in N_{\GL_2(\Z/N)}(G)$.  Thus $r_G(E) \leq r(\pm G)$, so by (a) equality holds.
\end{proof}

\subsection{Isogeny invariance}

In this section, having in section \ref{S1.6} understood our probability as a condition relating isogenous elliptic curves, we are led to the following theorem which relates the image of Galois for isogenous curves.  

\begin{thm} \label{thm:greenbergup-inart}
Let $\varphi \colon E \to E'$ be an isogeny of elliptic curves over a number field $K$.  Let $N \in \Z_{\geq 1}$, let $G \colonequals \overline{\rho}_{E,N}(\Gal_K) \leq \GL_2(\Z/N)$ and similarly $G'$ for $E'$.  Then the groups $\Gamma_G, \Gamma_{G'} \leq \GL_2(\Q)$ are conjugate in $\GL_2(\Q)$, the associated modular curves $Y_G$ and $Y_{G'}$ are isomorphic over $\Q$, and 
\[ [\SL_2(\Z):\Gamma_G]=[\SL_2(\Z):\Gamma_{G'}]. \]
\end{thm}

As mentioned in the introduction, the invariance of the index of the $p$-adic Galois representation under isogeny was already proven by Greenberg \cite[Proposition 2.1.1]{greenberg}, by a different argument.  

\begin{proof}
Without loss of generality, we may assume that $\varphi \colon E \to E'$ is given by a cyclic $N$-isogeny, so that the Galois image $G$ has 
\begin{equation} \label{eqn:GinupperN}
G \leq \begin{pmatrix} * & * \\ 0 & * \end{pmatrix}
\end{equation}
It follows from this group-theoretic statement that for every elliptic curve $A$ over $K$ whose mod $N$-Galois image is (conjugate to a) subgroup of $G$, there is an isogeny $\varphi \colon A \to A'$ (over $K$, with cyclic kernel of order $N$ generated by the point corresponding basis vector) such that the mod $N$-Galois image of $A'$ is a subgroup of $G'$.  Moreover, $\det G = \det G'$, since the determinant is the cyclotomic character and so its image only depends on (the roots of unity in) $K$.  Finally, the dual isogeny maps $\varphi\spcheck \colon E' \to E$, and similarly maps $G'$ to $G$.  In other words, the moduli problems attached to $G$ and to $G'$ are naturally equivalent, which gives an isomorphism $Y_G \xrightarrow{\sim} Y_{G'}$ of curves over their common field of definition $\Q(\zeta_N)^{\det G}=\Q(\zeta_N)^{\det G'}$.

From \eqref{eqn:GinupperN} we have
\begin{equation} \label{eqn:SL2zgamman} 
[\SL_2(\Z) : \Gamma_G] = [\SL_2(\Z) : \Gamma_0(N)][\Gamma_0(N) : \Gamma_G]. 
\end{equation}
Applying the isogeny $\varphi$ and swapping basis vectors acts by conjugation by the element $\nu=\begin{pmatrix} 0 & 1 \\ N & 0 \end{pmatrix}$ so that $\nu \Gamma_G \nu^{-1} = \Gamma_{G'}$.  Since $\nu$ normalizes the group $\Gamma_0(N)$, we have 
\begin{equation}
[ \Gamma_0(N): \Gamma_G] = [\nu \Gamma_0(N) \nu^{-1} : \nu \Gamma_G \nu^{-1}]
= [\Gamma_0(N) : \Gamma_{G'} ].
\end{equation}
Plugging this into \eqref{eqn:SL2zgamman} gives the result on indices.
\end{proof}

\begin{rmk}
We believe that Theorem \ref{thm:greenbergup-inart} should also follow more generally from the natural compatibilities satisfied by Shimura's theory of canonical models \cite[\S 6.7]{Shimura}.  The argument above gives a bit more information, namely that $\Gamma_{G'}$ is obtained from $\Gamma_G$ under conjugation by the Atkin--Lehner involution of $\Gamma_0(N)$.
\end{rmk}

In the next section, we will prove that for modular curves $Y_G$ such that $\Gamma_G$ is torsion free of genus zero, the asymptotic point count depends only on the index $[\SL_2(\Z):\Gamma_G]$; together with Theorem \ref{thm:greenbergup-inart} and the theorem of Katz (as refined in the previous section), this provides a concise explanation and ultimately a proof that the probability $P_m$ is positive for $m \geq 5$.

\subsection{Asymptotics} 

In this section, we prove Theorem \ref{thm:hsupgrade}.  We recall notation from section \ref{sec:notation}, and we prove the following weaker version first.  

\begin{thm} \label{thm:hsupgrade0}
Let $G \leq \GL_2(\Z/N)$ have $\det G=(\Z/N)^\times$, and suppose that $\Gamma_G$ is torsion free of genus zero and has no irregular cusps.  Let
\[ d(G) \colonequals \tfrac{1}{2}[\PSL_2(\Z):\Gamma_G]=\tfrac{1}{4}[\SL_2(\Z):\Gamma_G]. \]
Then $d(G) \in 6\Z_{\geq 1}$, and there exists $c(G) \in \R_{\geq 0}$ such that
\[ N_G(H) \colonequals \#\{E \in \scrE_{\leq H} : \overline{\rho}_{E,N}(\Gal_\Q) \lesssim G\} = c(G) H^{1/d(G)} + O(H^{1/e(G)}) \] 
as $H \to \infty$, where $e(G)=2d(G)$.
\end{thm}

As mentioned in the introduction, we follow an approach outlined in Harron--Snowden \cite[\S 5]{hs}. 

\begin{proof}
Our proof proceeds in four steps.

\emph{Step 1: universal curve}.  Let $Y_G$ be the curve over $\Q$ given by Proposition \ref{prop:theresacurve}.  We are given that $\Gamma_G$ (equivalently $Y_G$) has genus zero.  If $Y_G(\Q)=\emptyset$, then the theorem is trivially true taking $c(G)=0$.  So we may suppose $\#Y_G(\Q)=\infty$, in which case by choosing a coordinate $t$ we have $Y_G = \Spec \Q[t] \smallsetminus S \subseteq \A_\Q^1=\Spec \Q[t]$ where $S \subseteq Y_G$ is a finite set of closed points (stable under $\Gal_\Q$).  Since $\Gamma_G$ is torsion free, by Proposition \ref{prop:theresacurve}(iv), there is a universal curve of the form
\begin{equation}  \label{eqn:yftgt}
E_{G,\textup{univ}} \colon y^2 = x^3 + f(t)x + g(t) 
\end{equation}
where $f(t),g(t) \in \Q(t)$ (and regular away from $S$).  In particular, for every elliptic curve $E$ over $\Q$ such that $\overline{\rho}_{E,N}(\Gal_\Q) \lesssim G$, there exists $t_0 \in \Q \smallsetminus S$ such that $E$ is isomorphic to the curve $y^2 = x^3 + f(t_0)x+g(t_0)$.

Repeating carefully the argument of Harron--Snowden \cite[Proposition 3.2, second proof of Lemma 3.3]{hs} (given under more restrictive hypothesis, but using the fact that $\Gamma_G$ has no irregular cusps by hypothesis), after minimally clearing denominators we have $f(t),g(t) \in \Q[t]$ with $\gcd(f(t),g(t))=1$, and 
\begin{equation} 
3\deg f(t) = 2\deg g(t) = \deg(j) = [\PSL_2(\Z):\Gamma_G] = 2d(G)
\end{equation}
and moreover $12 \mid [\PSL_2(\Z):\Gamma_G]$.  In particular, $d(G) = \frac{1}{2}[\PSL_2(\Z):\Gamma_G] \in 6\Z_{\geq 1}$.  We now homogenize, letting $t=a/b$ and clearing denominators, giving
\begin{equation}  \label{eqn:EABab}
E_{A,B} \colon y^2 = x^3 + A(a,b)x + B(a,b) 
\end{equation}
with $A(a,b),B(a,b) \in \Z[a,b]$ satisfying 
\begin{equation} \label{eqn:degfgd}
\begin{aligned}
\deg A(a,b) &= \deg f(t) = \tfrac{2}{3}d(G) \\
\deg B(a,b) &= \deg g(t)=d(G)
\end{aligned}
\end{equation}

\emph{Step 2: principle of Lipschitz}.  In view of \eqref{eqn:EABab}, as a first step we count the number of integer points in the region
\begin{align} \label{region}
R(H) \colonequals \lbrace (a,b) \in \R^2 : \abs{A(a,b)} \leq (H/4)^{1/3} \text{ and } \abs{B(a,b)} \leq (H/27)^{1/2} \rbrace
\end{align}
as $H \to \infty$.

We claim that the region $R(H)$ is bounded.  By the above, the polynomials $f(t),g(t)$ are coprime, so 
\begin{equation}
\max(\abs{f(x)}^3,\abs{g(x)}^2) \geq \mu > 0
\end{equation}
is bounded below for all $x \in \R$.  From \eqref{eqn:degfgd}, we have
\begin{equation}
\begin{aligned}
\abs{A(a,b)} &= \abs{b^{2d(G)/3}f(a/b)} \\
\abs{B(a,b)} &= \abs{b^{d(G)}g(a/b)}
\end{aligned}
\end{equation}
we conclude that 
\[ H \geq \max_{a,b \in \R} (\abs{4A(a,b)^3},\abs{27B(a,b)^2}) \geq \mu\abs{b^{d(G)}} \]
so $b$ is bounded; a symmetric argument shows that $a$ is bounded.  

Being closed and bounded, the region $R(H)$ is compact.  Moreover, $R(H)$ has rectifiable boundary (defined by polynomials).  By the Principle of Lipschitz \cite{davenport}, the number of integral points in the region (\ref{region}) is given by its area up to an error proportional to the length of its boundary.  

Conveniently, the region $R(H)$ is \emph{homogeneous} in $H$: dropping parentheses to write $H^{1/2d(G)}=H^{1/(2d(G))}$, again from \eqref{eqn:degfgd} we have
\begin{equation} \label{eqn:H2dr1}
H^{1/d(G)} R(1) = R(H).
\end{equation}
Indeed, if $(a',b')=(H^{1/2d(G)} a,H^{1/2d(G)} b)$, then from
\[ \abs{A(a', b')} = (H^{1/2d(G)})^{2d(G)/3}\abs{A(a,b)} = H^{1/3}\abs{A(a,b)} \]
and similarly with $B$, we have $(a',b') \in R(H)$ if and only if $(a,b) \in R(1)$.  

Therefore,  
\begin{equation} \label{eqn:areaRHrg}
\begin{aligned}
\#(R(H) \cap \Z^2) &= \area(R(H)) + O(\len(\bd(R(H)))) \\
&= \area(R(1)) H^{1/d(G)} + O(H^{1/2d(G)})
\end{aligned}
\end{equation}
where the exponent on the error term follows from being the arclength of a 2-dimensional compact region with polynomial boundary.

We will use a slight refinement of this estimate which improves the error term, due to Huxley \cite{huxley} (and applied in our setting by Pomerance--Schaefer \cite[\S 4]{ed_carl_arxiv})---our boundary is defined by nonlinear polynomials, so the error term in \eqref{eqn:areaRHrg} arising from lattice points on the boundary can be improved to $O(H^{1/2d(G)-\delta})$ for some $\delta>0$.  

\emph{Step 3: sieving}.  We now apply a sieve to take care of local conditions: among the lattice points counted in the previous step, we want exactly those with $E \in \scrE$.  We first restrict the count of lattice points, then adjust the constant by finitely many local factors.  

First, the points $(a,b) \in \Z^2$ such that $4A(a,b)^3+27B(a,b)^2=0$ lie on a curve, which by standard estimates is $O(H^{1/2d(G)})$ so applying this condition does not change \eqref{eqn:areaRHrg}.  Second, for the points $(a,b) \in \Z^2$ such that $p \mid a$ and $p \mid b$, we have overcounted and we need to apply the correction factor $1-1/p^2$ for all primes $p$.  We say $(a,b) \in \Z^2$ is \defi{groomed} if $4A(a,b)^3+27B(a,b)^2 \neq 0$ and $\gcd(a,b)=1$.  A standard M\"obius sieve argument (see e.g.\ Harron-Snowden \cite[Proof of Theorem 5.5]{hs}) with the improved error term $O(H^{1/2d(G)-\delta})$, together with \eqref{eqn:areaRHrg}, gives
\begin{equation} \label{eqn:areaRHrgfixed2}
\begin{aligned}
&\#\{(a,b) \in R(H) \cap \Z^2 : \text{$(a,b)$ groomed} \} \\
&\qquad\qquad = \frac{\area(R(1))}{\zeta(2)}H^{1/d(G)} + O(H^{1/2d(G)}).
\end{aligned}
\end{equation}

We now consider local conditions imposed by minimal models.  Suppose that $q$ is a power of a prime $p$ such that $q^4 \mid A(a,b)$ and $q^6 \mid B(a,b)$ with $\gcd(a,b)=1$.  Recall $A(t,1)=f(t)$ and $B(t,1)=g(t)$ are coprime.  If $p \nmid b$, then $f(t)$ and $g(t)$ have a common root $a/b \in \Z/q$, so $q$ divides the nonzero resultant $\Res_t(f(t),g(t)) \in \Z$ of $f(t)$ and $g(t)$ with respect to $t$ \cite[Chapter 3]{CLO}; similarly, if $p \nmid a$ then $q$ divides the resultant of $A(1,u)=u^{\deg f} f(1/u)$ and $B(1,u)$.  Let $m$ be the least common multiple of these two resultants.  Applying the Sun Zu Theorem (CRT), we have shown that if $e \in \Z_{\geq 1}$ satisfies $e^4 \mid A(a,b)$ and $e^6 \mid B(a,b)$, then in fact $e \mid m$.

So let $E$ be an elliptic curve over $\Q$, and suppose $E$ has a $G$-level structure defined over $\Q$ in the sense of Proposition \ref{prop:theresacurve}(iv).  By \eqref{eqn:EABab}, we have $E \colon y^2 = x^3 + Ax + B$ where $A=A(a,b)$ and $B=B(a,b)$ for some $a,b \in \Z$ with $\gcd(a,b)=1$ (coming from $t=a/b \in \Q$ in lowest terms).  Then there exists a unique integer $e \in \Z_{\geq 1}$ such that the unique representative of $E$ in $\scrE$ is given by $y^2=x^3+A'x+B'$ where $A'=e^{-4}A(a,b) \in \Z$ and $B'=e^{-6}B(a,b) \in \Z$: namely, the largest positive integer $e$ such that $e^{12} \mid \gcd(A(a,b)^3,B(a,b)^2)$.  We call $e$ the \defi{minimality defect} of $(a,b)$.  By the previous paragraph, we have $e \mid m$.  Moreover, 
\[ e^{12} \height E = \max(\abs{4A^3},\abs{27B^2}) \]
so $(A,B) \in R(e^{12} H)$.  Recalling the discussion at the end of section \S \ref{sec:modellcur}, let
\[ N_G^{\square}(H) \colonequals \#\{(E,G\iota) : E \in \scrE_{\leq H} \text{ and } \overline{\rho}_{E,N}(\Gal_\Q) \lesssim_{\iota} G\} \]
count the number of pairs $(E,G\iota)$ where $E \in \scrE_{\leq H}$ and $G\iota$ is a $\Gal_\Q$-stable $G$-equivalence class of isomorphism $E[N](\Kbar) \to (\Z/N)^2$.  Running this argument in the other direction, we conclude that the count
\begin{equation} \label{eqn:emabreh}
\begin{aligned}
N_G^{\square}(H) &= \sum_{e \mid m} \#\{(a,b) \in R(e^{12}H) \cap \Z^2 \\
&\qquad\qquad\qquad : \text{$(a,b)$ groomed, minimality defect $e$}\}.
\end{aligned}
\end{equation}

Of course, the condition that $(a,b)$ has minimality defect $e$ is determined by congruence conditions on $a$ and $b$.  Let $\delta_e$ be the proportion of integers (congruence classes) satisfying this condition, so $0 \leq \delta_e \leq 1$ and $\sum_{e \mid m} \delta_e = 1$.  For each $e$ and each such congruence class, the principle of Lipschitz applies; summing of congruences classes then multiplies the asymptotic by the factor $\delta_e$.  Applying \eqref{eqn:areaRHrgfixed2}, from \eqref{eqn:emabreh} we conclude
\begin{equation} \label{eqn:emabreh000}
N_G^{\square}(H) = \frac{\area(R(1))}{\zeta(2)} \biggl(\sum_{e \mid m} \delta_e e^{12/d(G)}\biggr) H^{1/d(G)} + O(H^{1/2d(G)})
\end{equation}
so in particular
\begin{equation} 
c^{\square}(G) = \frac{\area(R(1))}{\zeta(2)} \biggl(\sum_{e \mid m} \delta_e e^{12/d(G)}\biggr).
\end{equation}

\emph{Step 4: automorphisms}.  Finally, to count the number of curves (rather than curves equipped with level structure), we apply Lemma \ref{g_structure_lem}.  For the curves with Galois image exactly $G$ (up to conjugacy) and $j(E) \neq 0,1728$, we have overcounted by the factor $2r(\pm G)=r(G)$, the additional factor $2$ coming from $t=a/b=a'/b' \in \Q$ is in lowest terms if and only if $(a',b')=\pm(a,b)$.  The curves with $j(E)=0,1728$ have $A=0$ or $B=0$, so are negligible (comparing to the length of the boundary).  For the remaining curves, suppose that $E$ has $\overline{\rho}_{E,N}(\Gal_\Q) = G' < G$ a proper subgroup (up to conjugation).  If $\Gamma_{G'}$ has genus $\geq 1$, then $N_{G'}(H)$ is either finite or grows slower than any power of $H$ (see Serre \cite[p.\ 133]{serre:lmw}), so in particular is $O(H^{1/d(G)})$.  Otherwise, $\Gamma_{G'}$ has genus zero and is still torsion free without irregular cusps.  Since $\det G' = \det G = (\Z/N)^\times$, we have
\[ [G:G']=[\Gamma_G:\Gamma_{G'}] \in \Z_{\geq 2} \]
so $d(G') \geq 2d(G)$.  Applying Step 3 then shows that the count of these curves is negligible.  

Thus from \eqref{eqn:emabreh000} we get that   
\begin{equation} 
\#\{E \in \scrE_{\leq H} : \overline{\rho}_{E,N}(\Gal_\Q) \lesssim G\} 
 = c(G) H^{1/d(G)} + O(H^{1/2d(G)})
 \end{equation}
where 
\begin{equation} \label{eqn:full_expression} 
c(G) = \frac{c^{\square}(G)}{r(G)} = \frac{\area(R(1))}{r(G)\zeta(2)} \biggl(\sum_{e \mid m} \delta_e e^{12/d(G)}\biggr)
\end{equation}
as claimed.
\end{proof}

\begin{cor} \label{cor:canremove}
With notation as in Theorem \textup{\ref{thm:hsupgrade0}}, we have
\[ N_G(H) = \#\{E \in \scrE_{\leq H} : \overline{\rho}_{E,N}(\Gal_\Q) \sim G\} + O(H^{1/2d(G)}). \]
\end{cor}

In other words, counting curves with image contained in $G$ is asymptotic to the count of curves with image equal to $G$.

\begin{proof}
Proven in Step 4 of the proof of Theorem \ref{thm:hsupgrade0}. 
\end{proof}

\begin{prop} \label{prop:effcomput}
The constant $c(G)$ in Theorem \textup{\ref{thm:hsupgrade0}} is effectively computable.
\end{prop}

\begin{proof}
We first claim that the universal curve is effectively computable, in the sense that there is a Turing machine (effective procedure) that, given input $G$, outputs $f(t),g(t) \in \Q(t)$ such that \eqref{eqn:yftgt} is universal.  We compactify $Y_G$ by adding cusps $X_G \colonequals Y_G \cup \Delta$; the set $\Delta$ (naturally identified with the set of $G$-orbits of $\PP^1(\Z/N)$) is effectively computable.  By Voight--Zureick-Brown \cite[Chapter 4]{vzb}, the canonical ring of $Y_G$ is the log canonical ring of $X_G$; this graded ring has a simple, explicit description \cite[\S 4.2]{vzb} in terms of $\#\Delta$.  Moreover \cite[\S 6.2]{vzb}, the log canonical ring is isomorphic to the graded ring of modular forms of even weight for $\Gamma_G$; by linear algebra with $q$-expansions computed via modular symbols as explained by Assaf \cite{assaf}, we obtain explicit equations for this canonical ring, realizing $X_G$ as a subvariety of weighted projective space.  Next, we can effectively determine if $X_G(\Q)=\emptyset$ and, if $X_G(\Q) \neq \emptyset$, compute $P_0 \in X_G(\Q)$: briefly, we compute a canonical divisor, embed $X_G \to \PP^2$ as a conic, and either find that $X_G(\Q_p)=\emptyset$ for some prime $p$ or we find a point in $X_G(\Q)$, after which we may parametrize the entire set $X_G(\Q)$ in terms of a parameter $t$, giving a computable isomorphism between the field of fractions of the log canonical ring and $\Q(t)$.  Finally, using linear algebra we recognize the Eisenstein series $E_4,E_6$ first as elements of the graded ring and then as rational functions in $t$.  

The remaining quantities are also effectively computable.  We compute the degree as 
\[ 4d(G)=[\SL_2(\Z) : \Gamma_G]=[\SL_2(\Z/N) : G \cap \SL_2(\Z/N)]\phi(N). \]  
For the constant $c(G)$, we note that the area $\area(R(1))$ can be computed to any desired precision by numerical integration, and $r(G)$ can be determined by finite exhaustion.  The integer $m$ is effectively computable as the least common multiple of resultants, and we can find the set of divisors of $m$ and then for each $e \mid m$, compute the proportion $\delta_e$ by exhaustive enumeration.  
\end{proof}

\begin{rmk}
Actually, by work of Sebbar \cite{sebbar} there are exactly $33$ torsion-free, genus zero subgroups of $\PSL_2(\Z)$, all of which lift to torsion-free subgroups of $\SL_2(\Z)$ by Kra \cite[Theorem, p.~181]{kra}.  Up to twist, there are only finitely many $G$ that can give each $\Gamma_G$, so the set of groups $G$ that satisfy the hypotheses of Theorem \ref{thm:hsupgrade} is finite (again, up to twist).  So it would be desirable to carry out the proof of Proposition \ref{prop:effcomput} in every case, and to just compute these constants (keeping track of the effect of the twist)---but such a task lies outside of the motivation and scope of this paper.  

Nevertheless, many of the curves in Sebbar's list arise in our analysis, as follows.  By (\ref{eqn:GammaG}) and the natural projection $\SL_2(\Z) \to \PSL_2(\Z)$, we can associate to every $G_\ell(n;r,s)$ a subgroup $\Gamma_G$ of $\PSL_2(\Z)$ via
$$
G_\ell(n;r,s) \leftrightarrow \Gamma_{\overline{G}_{\ell^n}(n;r,s)},
$$
(though of course a given group may not have genus 0).  Of the 33 genus zero subgroups of $\PSL_2(\Z)$, some can be written as $\Gamma_{G}$ with $G=\overline{G}_{\ell^n}(n;r,s)$.  In particular, we have 
\begin{equation}
\begin{aligned}
\Gamma(\ell^n) &= \Gamma_{\overline{G}_{\ell^n}(2n;n,n)} \\
\Gamma_1(\ell^n) &= \Gamma_{\overline{G}_{\ell^n}(n;n,0)} = \Gamma_{\overline{G}_{\ell^n}(n;0,0)} \\
\Gamma_0(4) &=\Gamma_{\overline{G}_4(2;1,0)},
\end{aligned}
\end{equation}
with functorial intersections.  In this way, the 16 groups
\begin{center}
\begin{tabular}{llllll}
$\Gamma(2)$ &$\Gamma(3)$ & $\Gamma(4)$ & $\Gamma(5)$ & $\Gamma_1(5)$ & $\Gamma_1(7)$\\
 $\Gamma_1(8)$ & $\Gamma_1(9)$ & $\Gamma_1(10)$ & $\Gamma_1(12)$ &$\Gamma_0(4)$ & $\Gamma_0(6)$ \\
$\Gamma_0(4) \cap \Gamma(2)$ & $\Gamma_1(8) \cap \Gamma(2)$ &  $\Gamma_0(2) \cap \Gamma(3)$ & $\Gamma_0(3) \cap \Gamma(2)$ 
\end{tabular}
\end{center}
in \cite{sebbar} can be each realized as (intersections of) the $\Gamma_{\overline{G}_\ell(n;r,s)}$.  Of the remaining 17 torsion-free genus zero groups, 9 can be realized as $\Gamma_H$, where $H$ is a \emph{proper} subgroup of some $\overline{G}_\ell(n;r,s)$.
The remaining 8 torsion-free genus zero groups
\begin{align*}
&\Gamma_0(8),\ \Gamma_0(9),\ \Gamma_0(8) \cap \Gamma(2),\ \Gamma_0(12),\\
&\Gamma_0(16),\ \Gamma_0(18),\ \Gamma_0(16) \cap \Gamma_1(8),\ \Gamma_0(25)\cap\Gamma_1(5)
\end{align*}
do not correspond to a $G_\ell(n;r,s)$ (or to an intersection).
\end{rmk}

We now officially conclude the proof.

\begin{proof}[Proof of Theorem \textup{\ref{thm:hsupgrade}}]
Combine Theorem \ref{thm:hsupgrade0} with Proposition \ref{prop:effcomput}.
\end{proof}

\section{The probabilities \texorpdfstring{$P_m$}{Pm} for \texorpdfstring{$m \geq 5$}{m at least 5}} \label{sect57}

In this section, we prove Theorem \ref{mainresult}, and we obtain an explicit result for the cases $m=5$ and $m=7$.  Although we do not do so, one can apply the arguments of this section to similarly compute $P_m$ for the remaining values of $m \geq 5$.    

\subsection{Proof of main result} \label{sec:generalpm_jv} 

Let $m \in \lbrace 1,2,\dots,10,12,16 \rbrace$.  As in section \ref{sec:notation}, we seek to refine our understanding of the subset
\begin{equation} 
\scrE_{m?} \colonequals \{E \in \scrE : \text{$m \mid \#E(\F_p)$ for a set of primes $p$ of density $1$}\}
\end{equation}
by considering the probability 
\begin{equation}  \label{eqn:limit_eqn3}
P_{m} \colonequals \lim_{H \to \infty} \frac{\#\{E \in \scrE_{\leq H} : m \mid \#E(\Q)\sbtors\}}{\#\{E \in \scrE_{m?} \cap \scrE_{\leq H}\}};
\end{equation}
in particular, we want to show $P_m$ is defined.  (Until we do, we may take $P_m$ to be the $\limsup$.)

We now proceed to prove Theorem \ref{mainresult} for $m \geq 5$.  Our strategy is as follows.  First, building on section \ref{S1.6}, we show that 100\% of curves in the numerator and denominator of $P_m$ are obtained from curves whose $\ell$-adic Galois image in a clean basis is equal to $G_\ell(n;r,s)$ for every $\ell \mid m$ (in particular, the mod $m$ image is the full preimage of the reductions modulo $\ell^n \parallel m$).  Second, using Theorem \ref{thm:hsupgrade0}, we give an asymptotic count for these curves; we find a positive proportion, as predicted by Theorem \ref{thm:greenbergup-inart}.

\begin{defn} \label{defn:z0}
We say that an elliptic curve $E$ over $\Q$ is \defi{$m$-full} if for all $\ell^n \parallel m$, there exist $r,s \in \Z_{\geq 0}$ with $r,s \leq n$ such that $\rho_{E,\ell}(\Gal_\Q)=G_\ell(n;r,s)$ (in a basis for $T_\ell(E)$).
\end{defn}

As in \eqref{eqn:swapbasis}, in Definition \ref{defn:z0} we may without loss of generality further suppose that $r+s \leq n$.  By Lemma \ref{lem:hasdenyup}, if $E$ is $m$-full, then $E \in \scrE_{m?}$.  The following proposition provides a converse sufficient for our purposes.

\begin{prop} \label{prop:countpointsonmodcrve}
We have 
\[ \#\{E \in \scrE_{\leq H} : \textup{$E$ is $m$-full}\} \sim \#(\scrE_{m?} \cap \scrE_{\leq H}) \]
as $H \to \infty$.
\end{prop}

\begin{proof}
Let $E \in \scrE_{m?}$.  By Corollary \ref{cor:sumitup}, there exists a cyclic isogeny $\varphi \colon E \to E'$ of degree $d \mid m$ such that for all $\ell^n \mid m$, we have $\rho_{E',\ell}(\Gal_\Q) \leq G_\ell(n;r,n-r)$ for some $0 \leq r \leq n$ (with these quantities depending on $\ell$).  Moreover, by Theorem \ref{upgrade}, if $\rho_{E',\ell}(\Gal_\Q)=G_\ell(n;r,n-r)$ for all $\ell \mid m$ (so equality holds), then $E$ is $m$-full.  

Let $T_m E \colonequals \varprojlim_n E[m^n](\Qbar) \simeq \prod_{\ell \mid m} \Z_\ell^2$ be the $m$-adic Tate module; let 
\begin{equation} 
\rho_{E,m} \colon \Gal_\Q \to \Aut_{\Z_m}(T_m E) \simeq \GL_2(\Z_m) \simeq \prod_{\ell \mid m} \GL_2(\Z_\ell) 
\end{equation}
be the associated Galois representation, and let $G \colonequals \rho_{E,m}(\Gal_\Q)$ be the image.  Repeat this with $E'$ and $G'$.  Let 
$G'_{\textup{full}} \colonequals \prod_{\ell \mid m} G_\ell(n;r,n-r)$, 
so by the first paragraph we have $G' \leq G'_{\textup{full}}$.  Suppose that $G' < G'_{\textup{full}}$ is a strict inequality; we will show the count of the curves $E$ obtained in this way is asymptotically negligible.  

We first consider the counts of the target curves $E'$.  We begin by reducing to a finite problem, lifting an argument from Sutherland--Zywina \cite[Proof of Proposition 3.6(i)]{SZ}.  The profinite group $G'_{\textup{full}}$, as a product of compact $\ell$-adic Lie groups, satisfies condition (iv) of a proposition of Serre \cite[Proposition, \S 10.6, p.~148; Example 1, p.~149]{serre:lmw}, so it satisfies condition (ii): its Frattini subgroup $\Phi(G'_{\textup{full}})$, the intersection of the maximal closed proper subgroups, is open.  Therefore there are only finitely many maximal proper open (so finite index) subgroups of $G'_{\textup{full}}$.  In particular, $G'$ is contained in (at least) one of these subgroups.  

Since $m \geq 5$, the group $\Gamma_{G'_{\textup{full}}}$ is torsion free by Lemma \ref{lem:quickpropr}(d), and so by Proposition \ref{prop:theresacurve}, there exists a curve $Y_{G'_{\textup{full}}}$ that is a fine moduli space for $G'_{\textup{full}}$.  Since $G' < G'_{\textup{full}}$, the same holds for $G'$ and moreover we have a map $Y_{G'} \to Y_{G'_{\textup{full}}}$.  Since $[G'_{\textup{full}}:G']>1$ and $\det(G')=\det(G'_{\textup{full}})=\Z_m^\times$, we have 
\[ [G'_{\textup{full}} \cap \SL_2(\Z_m) : G' \cap \SL_2(\Z_m)]>1, \] 
and hence $[\Gamma_{G'_{\textup{full}}}:\Gamma_{G'}]>1$.  Repeating the argument in Corollary \ref{cor:canremove}, the asymptotic count of elliptic curves parametrized by $Y_{G'}$ are negligible in comparison to those parametrized by $Y_G$ (the image of $Y_{G'}(\Q)$ in $Y_G(\Q)$ is thin).  Therefore
\begin{equation}  \label{eqn:Eph}
\begin{aligned}
&\#\{E' \in \scrE_{\leq H} : [E'] \in Y_{G'}(\Q)\} \\
&\qquad\qquad= o(\#\{E' \in \scrE_{\leq H} : [E'] \in Y_{G'_{\textup{full}}}(\Q)\}).
\end{aligned}
\end{equation}
Repeating this argument with $G'$ one of the finitely many maximal proper subgroups and summing, we have
\[ \#\{E' \in \scrE_{\leq H} : G' < G'_{\textup{full}}\} = o(\#\{E \in \scrE_{\leq H} : [(E,\iota)] \in Y_{G'_{\textup{full}}}(\Q)\}). \]

To finish, we count the curves $E$.  By Theorem \ref{thm:greenbergup-inart}, the groups $\Gamma_G$ is conjugate in $\GL_2(\Q)$ to $\Gamma_{G'}$ and have the same underlying modular curve.  By Theorem \ref{thm:hsupgrade0}, the asymptotics for the count of such $E$ is the same as that for counting $E'$; therefore, the result follows from \eqref{eqn:Eph}.
\end{proof}

With Proposition \ref{prop:countpointsonmodcrve} in hand, we just need to count by height the number of $m$-full elliptic curves by the choices for the groups $G_\ell(n;r,s)$ for $\ell^n \parallel m$ subject to \eqref{eqn:swapbasis}, and then to decide the proportion of which have $m$-torsion, as follows.  We recall the special case $\ell^n=2^1$ in Example \ref{exm:meq2doh}.

\begin{cor} \label{cor:pm5}
For $m \geq 5$, the probability $P_m$ is nonzero for all $m$.
\end{cor}

\begin{proof}
By Proposition \ref{prop:countpointsonmodcrve}, in the denominator of $P_m$ we need to count curves parametrized by groups $G \leq \GL_2(\Z/m)$ isomorphic (via the CRT) to the product $\overline{G}_{\ell^n}(n_\ell;r_\ell,s_\ell)$ (with $0 \leq r_\ell,s_\ell,r_\ell+s_\ell \leq n_\ell$, where $m = \prod \ell^{n_\ell}$, by \eqref{eqn:swapbasis}), and the numerator consists of the subset of counts with $r_\ell+s_\ell=n_\ell$.  By 
Lemma \ref{lem:quickpropr}(d), the groups $\Gamma_{G}$ are torsion free.  Only groups $G$ with $\det G=(\Z/m)^\times$ and $\Gamma_G$ of genus zero contribute nonnegligibly.  By Theorem \ref{thm:hsupgrade0}, the asymptotic for such a group is determined by $d(G)=\tfrac{1}{4}[\SL_2(\Z):\Gamma_{G}]$.  

We compute 
\begin{equation}
[\SL_2(\Z):\Gamma_{G}]=[\GL_2(\Z/m) : G] = \prod_{\ell \mid m} [\GL_2(\Z_\ell) : G_\ell(n_\ell;r_\ell,s_\ell)] 
\end{equation}
since the group is a direct product.  But we computed these indices in Lemma \ref{lem:quickpropr}: they only depend on whether $\min(r_\ell,n_\ell-r_\ell)\geq 1$ or not; the smallest degree $d(G)$ (from the smallest index, giving the largest asymptotic $H^{1/d(G)}$) occurs when $\min(r_\ell,n_\ell-r_\ell)\geq 1$ for each $\ell$.  Whatever the largest asmyptotic, we may always choose $s_\ell=n_\ell-r_\ell$ and by Lemma \ref{lem:charimage}(a) such curves have $m$-torsion, hence arise with positive probability.
\end{proof}

For the sake of explicitness, we indicate the rate of growth for each group in Table \ref{tab:allthedata}.  By a straightforward calculation in \textsf{Magma} \cite{magma}, we find Table \textup{\ref{tab:allthedata}}: the universal elliptic curve for $G_\ell(n;r,n-r)$ is isogenous to $G_\ell(n;r,n-r-k)$ for $k \leq n-r$ and $G_\ell(n;n-r,r-k)$ for $k \leq r$, so we can use universal equations for one to get to all others.  A list of all universal polynomials for the $m$-full groups that occur can be found online \cite{polysonline}.

\begin{equation} \label{tab:allthedata}\addtocounter{equation}{1} \notag
\begin{gathered}
{\renewcommand{\arraystretch}{1.1}
\begin{tabular}{c|c||c|c} 
$m$ & $G$ & $d(G)$ & \textup{torsion} \\[0.1ex]
\hline
\hline
$5$ & \textup{all} & $6$ & $\{0\}$, $\Z/5$ \\
$6$ & \textup{all} & $6$ & $\Z/2$, $\Z/6$ \\
$7$ & \textup{all} & $12$ & $\{0\}$, $\Z/7$ \\
$8$ & $G_2(3;r,0)$, $r=1,2,3$ & $12$ & $\Z/2^r$ \\
$8$ & $G_2(3;r,1)$, $r=1,2$ & $6$ & $\Z/2^r \times \Z/2$ \\
$9$ & \textup{all} & $18$ & $\{0\}$, $\Z/3$, $\Z/9$ \\
$10$ & \textup{all} & $18$ & $\Z/2$, $\Z/10$ \\
$12$ & $G_2(4;r,0) \times G_3(1;0,0)$, $r=1,2$ & $24$ & $\Z/2^r$ \\
$12$ & $G_2(4;1,1) \times G_3(1;1,0)$ & $12$ & $\Z/6 \times \Z/2$ \\
$16$ & \textup{all} & $24$ & $\Z/2^r \times \Z/2$, $r=0,1,2,3$
\end{tabular}} \\
\text{Table \textup{\ref{tab:allthedata}}: Data for modular curves parametrizing $m$-full elliptic curves}
\end{gathered}
\end{equation}

We find that for $m \in \{5,6,7,9,10,16\}$, all $m$-full groups $G$ have the same index $d(G)$; for $m=8,12$, we distinguish between two cases.  

\subsection{Setup to compute \texorpdfstring{$P_\ell$}{Pl} for \texorpdfstring{$\ell=5,7$}{l=5,7}} 

In the remainder of this section, we follow the proof of Corollary \ref{cor:pm5} and compute $P_\ell$ for $\ell=5,7$.  The main simplification in these cases is that, aside from a negligible subset when $\ell=5$ (see Lemma \ref{isog_lem}), elliptic curves in $\scrE_{\ell?}$ either have a global point of order $\ell$, or are $\ell$-isogenous to one that does.  

Let $\ell \in \lbrace 5,7 \rbrace$.  The Tate normal form of an elliptic curve, which gives a universal curve with a rational $\ell$-torsion point, has Weierstrass model
\begin{align} \label{tate_normal}
E(t): y^2 + (1-c)xy - by = x^3 - bx^2
\end{align}
with $b,c \in \Z[t]$ explicitly given; the rational point $(0,0)$ generates a rational subgroup of order $\ell$, and accordingly the image of the $\ell$-adic Galois representation lies in the group $G_\ell(1;1,0)$ as in Example \ref{exm:rn1od0s}.  Applying V\'elu's formulas \cite[(11)]{velu} to the isogeny with kernel generated by $(0,0)$, one obtains a model:
\begin{align} \label{tate_normal_isog}
E'(t): y^2 + (1-c)xy-by=x^3-bx^2 + dx + e,
\end{align}
for $d,e \in \Z[t]$.  The curve $E'(t)$ is the universal elliptic curve for the moduli problem of elliptic curves with $\ell$-adic Galois representation contained in $G_\ell(1;0,0)$, just as in Lemma \ref{lem:charimage}(b) the property that for any nonsingular specialization $t \in \Q$ it locally has a subgroup of order $\ell$.

Passing to short Weierstrass form, we write
\begin{align*}
E(t): y^2 &= x^3 + f(t)x + g(t) \\
E'(t): y^2 &= x^3 + f'(t)x+g'(t), 
\end{align*}
for explicit polynomials $f(t),f'(t),g(t),g'(t) \in \Q[t]$ given in \eqref{5polys} below.  Let $j(t)$ (resp.~$j'(t)$) be the $j$-function of $E(t)$ (resp.~$E'(t)$).  

Recall the integer $r(G)$ defined in \eqref{eqn:rG}.  As in Example \eqref{exm:GZN3}, we find that \begin{equation} \label{eqn:noRG}
r(G)=4,6
\end{equation} 
for $m=5,7$, so the ratio of the two cancels in each case.   

Writing $t = a/b$ and homogenizing, we finally arrive at two-parameter integral models 
\begin{equation} \label{int_models}
\begin{aligned} 
E(a,b): y^2 &= x^3+A(a,b)x + B(a,b) \\
E'(a,b): y^2 &= x^3+A'(a,b)x+B'(a,b),
\end{aligned}
\end{equation}
where $A,B \in \Z[a,b]$ and $A',B' \in \Z[a,b]$ are coprime pairs.  

We can now count integral curves by height and apply the methods of the previous sections.  Before preceding, as a guide to the reader we give an overview of the calculations in both cases here.

The probability $P_\ell$ will follow from the explicit computation of two growth constants: $c(G_\ell(1;1,0))$ and $c(G_\ell(1;0,0))$, associated to elliptic curves with a rational point of order $\ell$, and those that admit a rational $\ell$-isogeny (but not a rational point of order $\ell$), respectively.  Since the main growth terms have the same degree (see Table \ref{tab:allthedata}), we find
\begin{equation} \label{prob_definition}
P_\ell = \frac{c(G_\ell(1;1,0))}{c(G_\ell(1;1,0))+c(G_\ell(1;0,0))} = \left(1 + \frac{c(G_\ell(1;0,0))}{c(G_\ell(1;1,0))}\right)^{-1},
\end{equation}
where the constants $c(G_\ell(1;1,0))$ and $c(G_\ell(1;0,0))$ are defined in (\ref{eqn:emabreh000}).  

To ease notation, we abbreviate 
\begin{align*}
c& \colonequals c(G_\ell(1;1,0)) \\
c'& \colonequals c(G_\ell(1;0,0)).
\end{align*}
We also define the quantities arising in Step 3 of Theorem \ref{thm:hsupgrade0}, namely
\begin{equation} \label{eqn:mmp}
\begin{aligned}
m &\colonequals \lcm\bigl(\Res_t(f(t),g(t)), \Res_t(\breve{f}(t),\breve{g}(t)) \bigr) \\
m' &\colonequals \lcm\bigl(\Res_t(f'(t),g'(t)), \Res_t(\breve{f'}(t),\breve{g'}(t)) \bigr),
\end{aligned}
\end{equation}
where $\breve{h}(t) \colonequals  t^{\deg h}h(1/t)$ is the reciprocal polynomial of $h(t) \in \Q[t]$,
along with the corresponding correction ratios $\delta_e$ and $\delta_e'$ measuring the proportion of curves with minimality defect $e$ for $e \mid m,m'$, respectively, all as appearing in \eqref{eqn:emabreh000}.

Therefore, to compute $P_\ell$ we are reduced to computing the ratio 
\begin{equation} \label{eqn:cpc}
\frac{c'}{c} = \frac{\Area(R'(1)) \sum_{e\mid m'} \delta_e'e^{12/d(G)}}{\Area(R(1))\sum_{e\mid m} \delta_ee^{12/d(G)}}, 
\end{equation}
where
\begin{equation} \label{eqn:explicitAB}
\begin{aligned} 
R(1) &\colonequals \lbrace (a,b) \in \R^2 : \abs{A(a,b)} \leq 4^{-1/3} \text{ and } \abs{B(a,b)} \leq 27^{-1/2} \rbrace \\
R'(1) &\colonequals \lbrace (a,b) \in \R^2 : \abs{A'(a,b)} \leq 4^{-1/3} \text{ and } \abs{B'(a,b)} \leq 27^{-1/2} \rbrace.
\end{aligned}
\end{equation}

We compute the local corrections by finite search.  The ratio of areas has a remarkably simple expression, as follows.  By Lemma \ref{lem:samecurve}, we have $\Gamma_{\overline{G}_\ell(1;0,0)}=\Gamma_{\overline{G}_\ell(1;1,0)}$, i.e., the curves $E(t)$ and $E'(t)$ are universal curves over the same base modular curve.  Over $\Q(\zeta_\ell)$, the determinant of the mod $\ell$ Galois representation (the cyclotomic character) becomes trivial, so both of these curves solve the same moduli problem over $\Q(\zeta_\ell)$, and hence over $\C$. Since these two schemes represent the same functor, there is an isomorphism between them.  Both have base scheme a Zariski open in $\PP^1$ (say with variables $t$ and $t'$, respectively), so there exists a linear fractional transformation $\psi$ such that $t' = \psi(t)$.  Postcomposing with the $j$-function $X \to X(1) = \PP^1$, we conclude that there exists a linear fractional transformation $\psi$ such that
\begin{equation} \label{eqn:jvarphit}
j(\psi(t)) = j'(t).
\end{equation}
In concrete terms, given the $j$-invariants $j(t)$ and $j'(t)$ we compare zeroes and poles to explicitly compute the linear fractional transformation $\psi$.  By homogenizing $t$ to $(a,b)$ and  computing the effect on each variable, we get a change of variables mapping $R(1)$ bijectively onto $R'(1)$.  Therefore, the ratio
$$
\frac{\Area(R'(1))}{\Area(R(1))} 
$$
is the determinant of the change of variables matrix!  (In particular, this can be given exactly without needing it for the two areas themselves.) 

\subsection{The case \texorpdfstring{$\ell=5$}{l=5}} 

We now carry out the above strategy for $\ell=5$.  In the Tate normal form \eqref{tate_normal}, we compute $b=c=t$ (see also Garc\'ia-Selfa--Tornero \cite[Thm.~3.1]{gst}); applying V\'elu's formulas \cite[(11)]{velu} gives 
\begin{equation}
\begin{aligned}
d &=-5t^3-10t^2+5t,\text{ and}\\
e &=-t^5-10t^4+5t^3-15t^2+t
\end{aligned}
\end{equation}
in \eqref{tate_normal_isog}.  The Weierstrass coefficients and $j$-invariants of $E(t)$ and $E'(t)$ are given by 
\begin{equation} \label{5polys}
\begin{aligned}
f(t) &=-27(t^4-12t^3+14t^2+12t+1)\\
g(t) &= 54(t^6-18t^5+75t^4+75t^2+18t+1) \\
j(t) &= \frac{f(t)^3}{t^5(t^2-11t-1)}\\[10pt]
f'(t)&=-27(t^4+228t^3+494t^2-228t+1) \\
g'(t)&=54(t^6-522t^5-10005t^4-10005t^2+522t+1)\\
j'(t) &= \frac{f'(t)^3}{t(t^2-11t-1)^5}.
\end{aligned}
\end{equation}

\begin{lem} \label{isog_lem}
The curve $E'(t_0)$ defined by \textup{\eqref{tate_normal_isog}} has a rational $5$-torsion point if and only $t_0 \in \Q^{\times 5}$. 
\end{lem}

\begin{proof}
The discriminant of $E'(t)$ is $t(t^2 - 11t - 1)^5$, so $t=0$ is the only rational singular specialization.  By explicitly computing the 5-division polynomial of $E'(t)$ using the expressions in (\ref{5polys}), one can show that the 5-torsion field of $E'(t)$ has Galois group $F_{20}$ over $\Q(t)$ and is the splitting field of $x^5-t$ over $\Q(t)$.  For any non-zero specialization $t = t_0$, the mod 5 representation of $E'(t_0)$ is a subgroup of 
$$
\begin{pmatrix} * & * \\ 0 & 1 \end{pmatrix}.
$$
If, in addition, $E'(t_0)$ has a rational 5-torsion point, then the above Galois representation is diagonal, yet must have surjective determinant.  Thus, the 5-torsion field of $E'(t_0)$ is $\Q(\zeta_5)$ and so the polynomial $x^5 - t_0$ has a rational root, but does not split; \emph{i.e.}~$t_0$ is a rational 5th power.

Conversely, if $t_0=s^5$ then the point
\begin{align*} 
&(s^8 + s^7 + 2s^6 - 2s^5 + 5s^4 - 3s^3 + 2s^2 - s, \\
 &\quad s^{12} - s^{11} - s^{10} + s^8 - 10s^7 + 13s^6 - 11s^5 + 5s^4 - 3s^3 + s^2)
 \end{align*}
is a point of order $5$.  
\end{proof}

\begin{rmk} \label{isog_rem}
If $t_0 \in \Q$, then as in Example \ref{exm:isoggraph}, the isogeny class to which $E(t_0)$ belongs typically contains only two curves, $E(t_0)$ and $E'(t_0)$, linked by a 5-isogeny, with the two representations in Lemma \ref{serre_exercise} (contained in a Borel subgroup); see for example the isogeny class with LMFDB \cite{lmfdb} label \href{http://www.lmfdb.org/EllipticCurve/Q/38/b/}{\textsf{38.b}}.  However, if $t_0$ is a 5th power, then $E'(t_0)$ has a rational 5-torsion point, the mod 5 representation of $E'(t_0)$ is contained in the split Cartan subgroup 
\[ \begin{pmatrix} 1 & 0 \\ 0 & * \end{pmatrix} \leq \GL_2(\F_5), \] 
and $E'(t_0)$ admits two different rational 5-isogenies: for example, \href{http://www.lmfdb.org/EllipticCurve/Q/1342/b/}{\textsf{1342.b}}.

By the classification of possible images of mod 5 and mod 7 representations of Zywina \cite[Theorems 1.4, 1.5]{zywina}, this is a ``worst case scenario'' for the isogeny graph of a curve in $\scrE_{5?}$.  By comparison, for curves in $\scrE_{7?}$, all isogeny classes contain two curves linked by a $7$-isogeny.  See the recent preprint Chiloyan--Lozano-Robledo \cite{alvaro} on the classification of isogeny graphs of elliptic curves over $\Q$. 
\end{rmk}

We now compute the all-important change of coordinates $\phi(t)$ in \eqref{eqn:jvarphit}.  We write 
\[ u \colonequals (11+5\sqrt{5})/2 \approx 11.09 \in \R_{>0} \] 
so that $u$ and $-1/u$ are the roots of the quadratic polynomial $t^2-11t-1$.  We define the linear fractional transformation 
\begin{equation}
\psi(t) \colonequals \frac{ut+1}{t-u},
\end{equation}
mapping $u \to \infty$, $0 \to -1/u$, and $\infty \to u$.  It is routine to verify that $j(\psi(t))=j'(t)$.  

\begin{lem} \label{5jacobian}
With $R(1),R'(1)$ as defined in \eqref{eqn:explicitAB}, we have 
\begin{equation}
\frac{\Area(R'(1))}{\Area(R(1))} = \frac{1}{5}.
\end{equation}
\end{lem}

\begin{proof}
By the observation following  (\ref{eqn:jvarphit}), the ratio of areas is the determinant of the change of variables matrix mapping $R(1)$ bijectively onto $R'(1)$.  There is a pleasant,  visible symmetry in this case---one which gave this entire project momentum---so we are even more explicit in this case.   

Define the angle $\theta \colonequals \arctan(2/11)/2$, so that 
$$
\cos \theta = \frac{1}{5} \sqrt{\frac{25+11\sqrt{5}}{2}} \quad \text{and} \quad \sin \theta = \frac{1}{5} \sqrt{\frac{25-11\sqrt{5}}{2}}.
$$
Direct calculation reveals that 
\begin{align*}
A'(a\cos \theta - b\sin \theta, a\sin \theta + b\cos \theta) &= A(\sqrt{5}a,-\sqrt{5}b), \\
B'(a\cos \theta - b\sin \theta, a\sin \theta + b\cos \theta) &= B(\sqrt{5}a,-\sqrt{5}b).
\end{align*}
In other words, a rotation by $\theta$, followed by a reflection and a scaling of $a$ and $b$ by $\sqrt{5}$ maps $R'(1)$ bijectively onto $R(1)$, as in Figure \ref{fig:agraphbgraph}.  

\begin{equation} \label{fig:agraphbgraph}\addtocounter{equation}{1} \notag
\begin{gathered}
\includegraphics[scale=0.28]{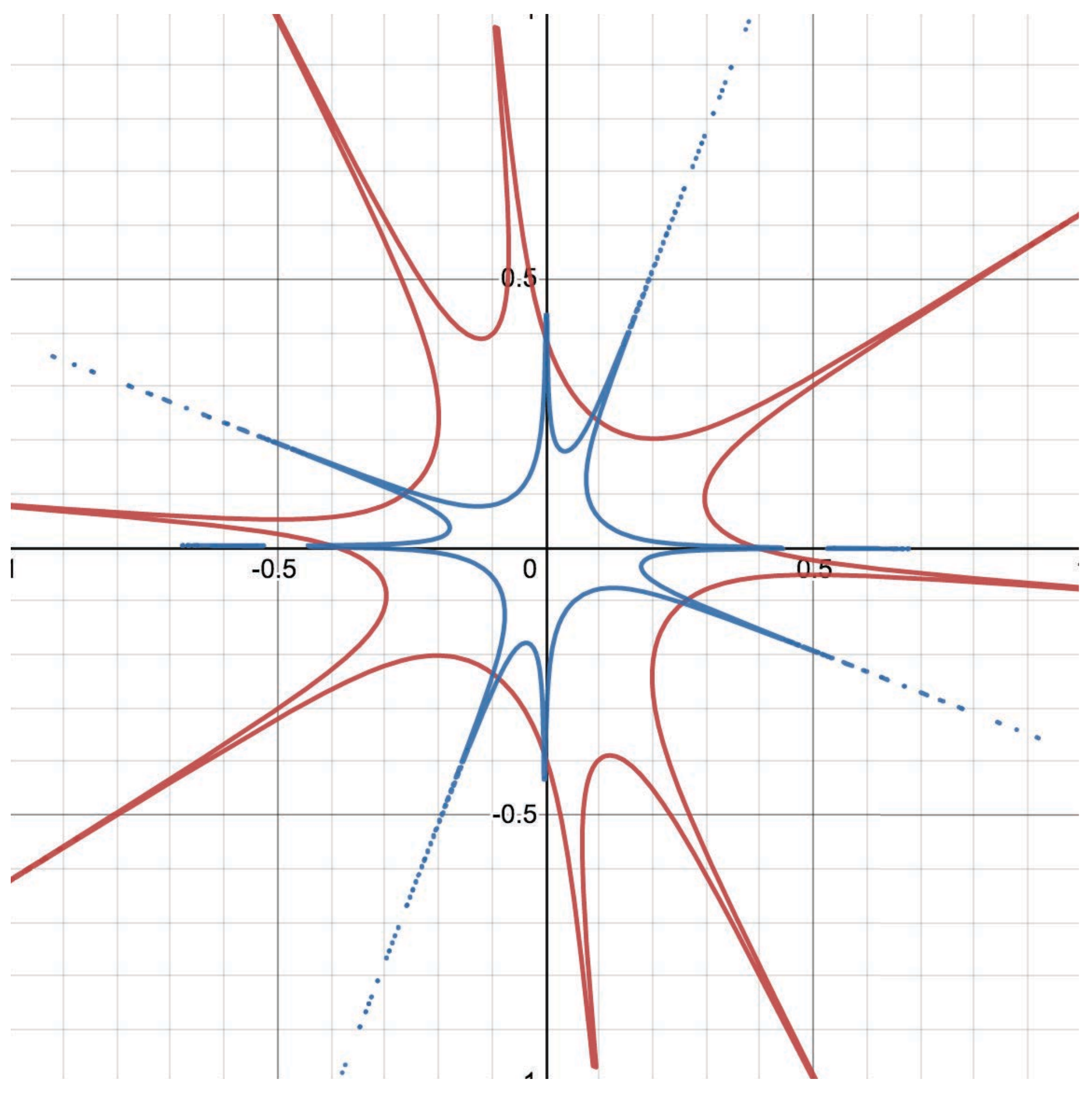} \qquad  \includegraphics[scale=0.28]{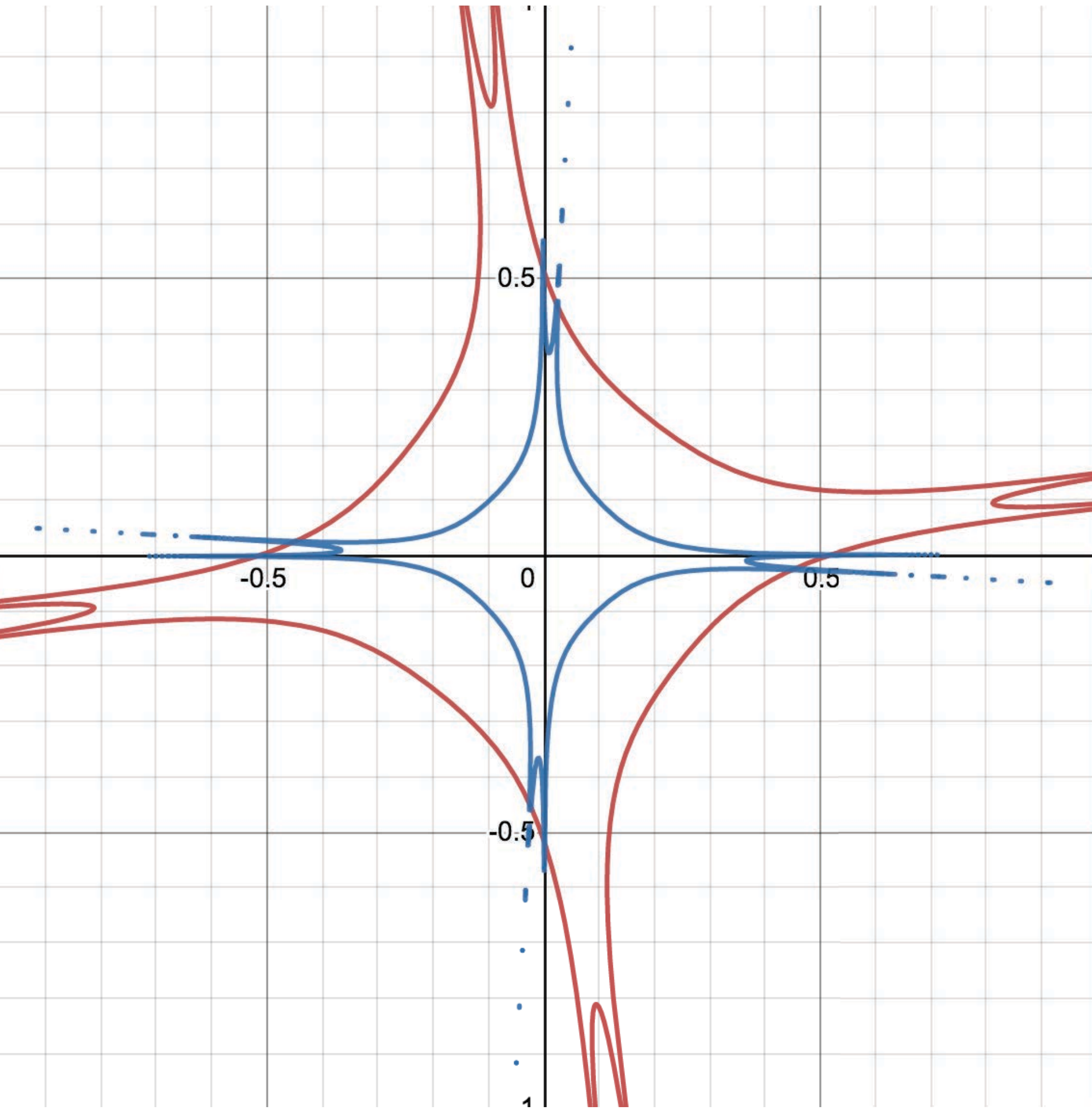} \\
\parbox[c]{4.5in}{\centering Figure \ref*{fig:agraphbgraph}: Symmetry of $R(1)$ and $R'(1)$, $m=5$}
\end{gathered}
\end{equation}

The ratio $\Area(R'(1))/\Area(R(1)) = 1/5$ follows from the fact that a reflection/rotation is area-preserving and the scaling is by $\sqrt{5}$ in both directions $a$ and $b$. 
\end{proof}

Now we calculate the sieve factor in (\ref{eqn:cpc}), which is the last ingredient needed for an exact expression for $P_5$.  Because we will perform a similar sieve for the case $P_7$, we go into some detail here and then proceed more quickly through this step when $\ell=7$.

We start by recording the integers $m$ and $m'$ defined in \eqref{eqn:mmp}, computed by resultants:
\begin{align*}
m &= 2^{16}3^{36}5, \\
m' &= 2^{16}3^{36}5^{25}.
\end{align*}
We recall that $d(G_5(1;1,0)) = d(G_5(1;0,0)) = 6$.  Thus, the local correction factors for $c$ is
\[ \sum_{e\mid m} \delta_ee^{12/d(G)} = \sum_{e\mid 2^{16}3^{36}5} \delta_ee^{2} \]
and similar for $c'$.  We now compute the $\delta_e$ and $\delta_e'$.

\begin{lem} \label{delta_e}
With all notation as above, we have $\delta_1=1$ and $\delta_e=0$ for all other divisors $e$ of $m$.
\end{lem}

\begin{proof}
Let $(a,b) \in \Z^2$.  One can easily verify by hand or computer that if $p=2, 3$, or $5$, then 
\begin{align*}
p^4 \mid A(a,b) \text{ and } p^6 \mid B(a,b)
\end{align*}
if and only if $a\equiv b \equiv 0 \pmod{p}$.  If $e>1$ is a possible minimality defect (recall this means that $e$ is the largest positive integer such that $e^{12} \mid \gcd(A(a,b)^3,B(a,b)^2)$), then $e$ is divisible by at least one of 2, 3, or 5 and so by the previous observation $(a,b)$ is not groomed.  Thus $\delta_e = 0$.  Since $\sum_{e|m} \delta_e = 1$, we have $\delta_e=1$.
\end{proof}

\begin{cor} \label{cor:cg110_5}
We have
\[
c = c(G_\ell(1;1,0)) = \frac{\Area(R(1))}{\zeta(2)}.
\]
\end{cor}

\begin{proof}
Plug Lemma \ref{delta_e} and \eqref{eqn:full_expression}.
\end{proof}

\begin{lem} \label{delta_e'}
We have $\delta'_e = 0$ for all divisors $e$ of $m'/5^{25}$.
\end{lem}

\begin{proof}
A similar calculation as in Lemma \ref{delta_e} shows that if $p=2,3$, then 
\begin{align*}
p^4 \mid A'(a,b) \text{ and } p^6 \mid B'(a,b)
\end{align*}
if and only if $a\equiv b \equiv 0 \pmod{p}$, so the same conclusion holds.
\end{proof}

By Lemma \ref{delta_e'}, it only remains to compute $\delta_e'$ for $e\mid 5^{25}$.

\begin{lem} \label{delta_5'}
We have $\delta_e' = 0$ for all $e >5$, $\delta_1' = 29/30$,  and $\delta_5' = 1/30$, so that 
\[
\sum_{e \mid 5^{25}} \delta_e' e^{2} =  \frac{29}{30} + \frac{1}{30}\cdot 25 = \frac{9}{5}.
\]
\end{lem}

\begin{proof}
We first show that $\delta_e' = 0$ for $e > 5$.  Suppose $e >5$, so that $e = 5^k$ for $1 \leq k \leq 25$.  Let $(a,b) \in \Z^2$ have minimality defect $e$, so $e^{12} \mid \gcd(A'(a,b)^3,B'(a,b)^2)$; then $e^{12} \mid A'(a,b)^3$, whence $e^{4}\mid A'(a,b)$.  Note that $e$ is divisible by $5^2$.  

Suppose further that $\gcd(a,b) = 1$. We will show that there are no solutions to the congruence 
\begin{align} \label{mod5^6}
A'(a,b) \equiv 0 \pmod{5^6},
\end{align}
implying that it is impossible for $e^4 \mid A'(a,b)$ unless $\gcd(a,b) \ne 1$, i.e., unless $(a,b)$ is not groomed.  

Since $\gcd(a,b) = 1$, either $a$ or $b$ is coprime to 5.  By the symmetry of the coefficients of $A'(a,b)$, it suffices to assume $b$ is coprime to 5, and hence invertible modulo $5^6$.  Multiplying (\ref{mod5^6}) by $(1/b)^6$, we are left with the congruence 
\[
f'(t)  \equiv 0 \pmod{5^6},
\]
to which there are no solutions.  As we sketched above, this is enough to conclude that $\delta_e' = 0$ for $e > 5$.  It remains to calculate $\delta_5'$, which we can do working modulo $25$.

Suppose 
\begin{align*}
A'(a,b) \equiv 0 \pmod{5^4} \quad\text{and}\quad
B'(a,b) \equiv 0 \pmod{5^6}.
\end{align*}
If $b$ is invertible modulo 25, then we are left to consider the congruences 
\begin{align*}
f'(t) \equiv 0 \pmod{5^4} \quad\text{and}\quad
g'(t) \equiv 0 \pmod{5^6}.
\end{align*}
which happens if and only if $t \equiv 18 \pmod{25}$.  Similarly if $a$ is invertible, then we find $t \equiv 7 \pmod{25}$.  We also note that 18 and 7 are inverses modulo 25, reflecting the fact that $A'$ and $B'$ are each reciprocal polynomials.  This accounts for 1/30 of the possible ratios $(a:b)$ among groomed $(a,b)$ modulo 25.  (Alternatively, working over $\PP^1(\Z/25)$, we find that of the 30 rational points, only $[18:1] = [1:7]$ solve the above congruences.)

Thus, $\delta_5' = 1/30$, and $\delta_1' = 1- 1/30 = 29/30$, and the correction factor of $9/5$ follows.
\end{proof}

\begin{cor} \label{cor:cg110_5'}
The constant $c'$ is given by
\[
c' = \frac{9\Area(R(1))}{5\zeta(2)}.
\]
\end{cor}

\begin{proof}
This follows immediately from Lemmas \ref{delta_e'} and \ref{delta_5'}, as in Corollary \ref{cor:cg110_5}.
\end{proof}

We finally arrive at the exact value of $P_5$.

\begin{cor}
We have $P_5 =  25/34 \approx 73.5\%$.
\end{cor}

\begin{proof}
Combining Corollaries \ref{cor:cg110_5} and \ref{cor:cg110_5'}, together with Lemmas \ref{5jacobian} and \ref{delta_5'}, we have
$$ 
\frac{c'}{c} = \frac{\Area(R'(1)) \sum_{e\mid m'} \delta_e'e^{2}}{\Area(R(1))} = \frac{1}{5} \cdot \frac{9}{5} = \frac{9}{25}. 
$$
By (\ref{prob_definition}), we have 
\[
P_5 = \frac{c}{c+c'} = \frac{1}{1 + c'/c} = \frac{1}{1+9/25}=\frac{25}{34}.  \qedhere
\]
\end{proof}

\begin{rmk} \label{5corroborate}
We perform a count of $5$-full elliptic curves $E \in \scrE_{5?} \cap \scrE_{\leq H}$ in \textsf{Magma} of height $H \leq 10^{36}$, giving 196772 with a global subgroup of order 5 and 70784 with only a local, but not global, subgroup of order 5.  These 70784 further decompose as 37944 with $e=1$ and 32840 with $e=5$.  These proportions are in line with the ones predicted above and altogether give a ratio of 
$$
196772/(196772+70784) \approx 73.5\%,
$$
which agrees nicely with our calculations above. 
\end{rmk}

\subsection{The Case \texorpdfstring{$\ell=7$}{l=7}} 

Repeating the steps in the previous section, we are more brief.  
The universal models for those curves have Weierstrass data:
\begin{align*}
f(t)&= -27t^8 + 324t^7 - 1134t^6 + 1512t^5 - 945t^4 + 378t^2 - 108t - 27 \\
g(t) &= 54t^{12} - 972t^{11} + 6318t^{10} - 19116t^9 + 30780t^8 - 26244t^7 + 14742t^6 \\
&\quad\ - 11988t^5 + 9396t^4 -  2484t^3 - 810t^2 + 324t + 54 \\
f'(t) &= -27t^8 - 6156t^7 - 1134t^6 + 46872t^5    - 91665t^4 + 90720t^3 - 44982t^2 \\
&\quad\ + 6372t - 27 \\
g'(t) &= 54t^{12} - 28188t^{11} - 483570t^{10} + 2049300t^9 - 3833892t^8 + 7104348t^7 \\
&\quad\ - 13674906t^6 + 17079660t^5 - 11775132t^4 + 4324860t^3 - 790074t^2 \\
&\quad\ + 27540t + 54.
\end{align*}
As above, we let $A,B,A',B'$ denote the homogenizations of $f,g,f',g'$.

The $j$-invariant $j(t)$ of $E(t)$ is given explicitly by 
$$
j(t) = \frac{(t^2 - t + 1)^3(t^6 - 11t^5 + 30t^4 - 15t^3 - 10t^2 + 5t + 1)^3}{t^7(t-1)^7(t^3 - 8t^2 + 5t + 1)}
$$
and so has simple poles at the roots of the polynomial  $h(t) \colonequals t^3-8t^2+5t+1$, and poles of order 7 at $0,1,\infty$.  Similarly, the $j$-invariant $j'(t)$ of $E'(t)$ has simple poles at 0, 1, $\infty$ and poles of order 7 at the roots of $h(t)$.  The roots of $h(t)$ are real, generate the field $\Q(\zeta_7 + \zeta_7^{-1})$, and we label them according to the ordering $\rho_1 < \rho_2 < \rho_3$. 
Under the linear fractional transformation 
$$
\psi(t) = \frac{(\rho_2-\rho_1)t+(\rho_1-\rho_2)\rho_3}{(\rho_2 -\rho_3)t+(\rho_1\rho_3-\rho_1\rho_2)},
$$
we have $j(\psi(t)) = j'(t)$.  We now proceed exactly as above.

\begin{lem} \label{conjecture}
With $R(1),R'(1)$ as defined in \eqref{eqn:explicitAB}, we have 
\begin{equation}
\frac{\Area(R'(1))}{\Area(R(1))} = \frac{1}{\sqrt{7}}.
\end{equation}
\end{lem}

\begin{proof}
The change of variables $(a,b) \mapsto J(a,b)$ defined by matrix multiplication (on columns)
$$
J \colonequals \begin{pmatrix} u & 0 \\ 0 & u \end{pmatrix} \begin{pmatrix}\rho_2-\rho_1 & \rho_1\rho_3-\rho_2\rho_3 \\ \rho_2 -\rho_3 & \rho_1\rho_3-\rho_1\rho_2\end{pmatrix},
$$
and $u = 7^{-3/4}$ maps $R(1)$ bijectively onto $R'(1)$ by checking that 
\begin{align*}
A(J(a,b)) = A'(a,b) \text{ and } B(J(a,b)) = B'(a,b).
\end{align*}
Then
$$
\abs{\det (J)} = \abs{u^2(\rho_2 - \rho_3)(\rho_1^2 -\rho_1\rho_2 - \rho_1\rho_3 + \rho_3\rho_2)} = \sqrt{7}, 
$$
which proves the lemma.
\end{proof}

Just like in the case $\ell=5$ we must compute the local correction factors.  We compute $m=-2^{32}3^{72}7$ and $m'=-2^{32}3^{72}7^{49}$ and check that $d(G) = 12$ (so that we sum $\delta_e e^1$ and $\delta_e' e^1$ over the divisors of $m$ and $m'$, respectively). 

\begin{lem} \label{bad7primes1}
We have $\delta_e = \delta_e' = 0$ when $e$ is divisible by a power of $2$, and $\delta_7 = 0$. 
\end{lem}

\begin{proof}
Similar to our work in the $\ell=5$ case, one checks that 
\begin{align*}
2^4 \mid A(a,b) \quad\text{and}\quad 2^6 \mid B(a,b)
\end{align*}
if and only if $a\equiv b \equiv 0 \pmod{2}$ and similarly for $A'(a,b)$ and $B'(a,b)$.  An analogous  calculation gives us the same result for the polynomials $A(a,b)$ and $B(a,b)$ when $p=7$.  

Thus, none of the groomed pairs $(a,b)$ contribute to the correction factor in these cases and by the identical argument to the one in Lemma \ref{delta_e} we conclude that $\delta_e = \delta_e'=0$ when $e$ is a power of 2 and that $\delta_7=0$, as claimed.
\end{proof}

\begin{lem} \label{bad7primes2}
With all notation as above, we have 
\begin{align*}
\delta_e'&=0 \text{ if } 7^2 \mid e. \\
\delta_e = \delta_e' &= 0 \text{ if } 3^2 \mid e \\
\delta_3 &= 1/4 \\
\delta_1 &= 3/4 \\
\delta_3' &= 7/32 \\
\delta_7' &= 3/32 \\
\delta_{21}' &= 1/32 \\
\delta_1' &= 21/32.
\end{align*}
\end{lem}

\begin{proof}
The reasoning is identical to that in Lemma \ref{delta_5'}.  Suppose $3^2\mid e$.  Let $(a,b) \in \Z^2$ have minimality defect $e$, so $e^{12} \mid  \gcd(A(a,b)^3,B(a,b)^2)$; then $e^{12} \mid A(a,b)^3$, whence $e^{4} \mid A(a,b)$.  Since $e$ is divisible by $3^2$, we have that $A(a,b) \equiv 0 \pmod{3^8}$.

If, in addition, $\gcd(a,b) = 1$ we can show that there are no solutions to the congruence 
\begin{align} \label{mod3^5}
A(a,b) \equiv 0 \pmod{3^5},
\end{align}
implying that it is impossible for $e^4 \mid A(a,b)$ unless $\gcd(a,b) \ne 1$, \emph{i.e.}~unless $(a,b)$ is not groomed.  

Since $\gcd(a,b) = 1$, either $a$ or $b$ is coprime to 3.  Moreover, since the coefficients of $A(a,b)$ are not symmetric, we must consider both cases.  If $b$ is coprime to 3, then it is invertible modulo $3^5$ and so we led to the congruence 
\[
f(t) \equiv 0 \pmod{3^5},
\]
which has no solutions. Similarly, we can invert $a$ to arrive at the congruence 
\[
t^8f(1/t) \equiv 0 \pmod{3^5},
\]
which also has no solutions. We can repeat this same argument for the polynomial $A'(a,b)$ and again for the polynomial $A'(a,b)$ when $e$ is divisible by $7^2$.  In all cases we conclude that there are no groomed pairs $(a,b)$ giving rise to divisibility by $e^{12}$ in these cases.  This leaves only a handful of cases left to work out: $\delta_3$, $\delta_3'$, $\delta_7'$, and $\delta_{21}'$.  These will, in turn, give us $\delta_1$ and $\delta_1'$.  

For $\delta_3$ and $\delta_3'$ we work in $\PP^1(\Z/3)$ and find 
\begin{align*}
A(a,b) \equiv 0 \pmod{3^4} \text{ and } B(a,b) \equiv 0 \pmod{3^6}
\end{align*}
if and only if 
\begin{align} \label{delta3congruences}
[a:b] \equiv [2:1] \equiv [1:2] \pmod{3}
\end{align}
This accounts for $1$ of the 4 points of $\PP^1(\Z/3)$, hence $\delta_3 = 1/4$.  For $\delta_3'$, we find the exact conditions as (\ref{delta3congruences}).

Turning to $\delta_7'$ we find 
\begin{align*}
A'(a,b) \equiv 0 \pmod{7^4} \text{ and } B'(a,b) \equiv 0 \pmod{7^6}
\end{align*}
if and only if 
\begin{align} \label{delta7congruences}
[a:b] \equiv [5:1] \equiv [1:3] \pmod{7}.
\end{align}
This accounts for 1 of the 8 points of $\PP^1(\Z/7)$.

For $\delta_{21}'$ we use the CRT combined with the proportions at 3 and 7 above.  We finally arrive at
\begin{align*}
\delta_3' &= \frac{1}{4} \cdot \frac{7}{8} = \frac{7}{32} \\
\delta_7' &= \frac{3}{4} \cdot \frac{1}{8} = \frac{3}{32} \\
\delta_{21}' &= \frac{1}{4} \cdot \frac{1}{8} = \frac{1}{32}.\\
\delta_1' &= 1 - \delta_3' - \delta_7' - \delta_{21}' = \frac{21}{32}. \qedhere
\end{align*}
\end{proof}

\begin{cor} \label{P7cor}
We have 
\[ P_7= 4/(4+\sqrt{7}) \approx 60.2\% \]
\end{cor}

\begin{proof}
By (\ref{prob_definition}) and Lemmas \ref{conjecture}, \ref{bad7primes1}, and \ref{bad7primes2}, we have 
$$
P_7 = \frac{c}{c+c'} = \frac{1}{1 + c'/c},
$$
with
$$ 
\frac{c'}{c}= \frac{1}{\sqrt{7}} \cdot \frac{ (21/32) + (7/32) \cdot 3 + (3/32) \cdot 7 + (1/32) \cdot 21}{ (3/4) + (1/4) \cdot 3} = \frac{\sqrt{7}}{4},
$$
from which the exact value of $P_7$ follows.
\end{proof}

\begin{rmk} \label{7corroborate}
Similar to \ref{5corroborate}, we perform a count of elliptic curves in \textsf{Magma} of height $H \leq 10^{72}$.  We get 1291676 with a global subgroup of order 7 (where 645918 correspond to $e=1$ and  645758 to $e=3$).  We also get 854432 that  locally have a subgroup of order 7, but not globally.   These 
854432 break down as: 213522 with $e=1$; 213704 with $e=3$;  213714 with $e=7$; and  213492 with $e=21$.  All of these proportions agree nicely with the predictions above, and give a ratio of 
$$
\frac{1291676}{1291676 + 854432} \approx 60.2\%,
$$
which is very good corroborating evidence for Corollary \ref{P7cor}.
\end{rmk}

\section{The probabilities \texorpdfstring{$P_3$}{P3} and \texorpdfstring{$P_4$}{P4}} \label{34div}

In this section, we compute the values of $P_3$ and $P_4$ using similar methods as in the previous section, but without appealing to the general result (in particular, there are non-fine moduli spaces).  In the case $m=3$ we can evaluate $P_3$ without computing explicit growth constants thanks to a symmetry argument, while for $m=4$ we express $P_4$ as a ratio of growth constants given explicitly by an integral. 

\subsection{Universal models}

Here we parametrize curves  that locally have a subgroup of order $m$ for $m \in \lbrace 3,4 \rbrace$, working in a bit more generality.  Let $F$ be a global field with $\opchar F \neq 2,3$ and let $E \colon y^2 = f(x) = x^3 + Ax + B$ be an elliptic curve over $F$.  For $d \in F^\times$, let $E_d \colon dy^2=f(x)$ denote the quadratic twist by $d$.  

\begin{lem} \label{E3weierstrass}
Suppose that $E$ locally has a subgroup of order $3$, i.e., $3 \mid \#E(\F_\frakp)$ for a set of primes $\frakp$ of $F$ of density $1$.  Then the following statements hold.
\begin{enumalph}
\item Either $E(F)[3] \neq \{\infty\}$ or $E_{-3}(F)[3] \neq \{\infty\}$.
\item There exist $a,b \in F$ and $u \in \{1,-3\}$ such that $E$ is defined by the equation
\[ y^2 = x^3 + u^2(6ab+27a^4)x + u^3(b^2-27a^6). \]
\end{enumalph}
\end{lem}

\begin{proof}
Let $E \in \scrE_{3?}$ be given by $y^2 = x^3+Ax+B$. By Lemma \ref{serre_exercise}, either $E(F)[3] \neq \{\infty\}$ or $E$ admits a $3$-isogeny over $F$ to a curve $E'$ with $E'(F)[3] \neq \{\infty\}$.  In either case, $E$ has a rational $3$-isogeny and the $x$-coordinate of a generator of the kernel must be defined over $\Q$.  Hence the $3$-division polynomial 
of $E$ has a root $a\in F$.  

By Theorem \ref{upgrade}, the semisimplification of the mod $3$ Galois representation attached to $E$ has $\overline{\rho}_{E,3}^{\textup{ss}} \simeq \mathbf{1} \oplus \epsilon_3$, where $\epsilon_3$ is the mod $3$ cyclotomic character.  
If $E$ has a 3-torsion point then 
$$
a^3 + Aa +B \in F^{\times 2}
$$
so we interpret $F(\mathbf{1}) = F(\sqrt{a^3+Aa+B})$.  Since $F(\epsilon_3) = F(\sqrt{-3})$, it follows that 
$$
F(\epsilon_3) = F(\sqrt{-3(a^3+Aa+B)}).
$$
Thus, either $E$ has a rational point of order 3 or its quadratic twist by $-3$ does, proving (a).  

Part (b) is by a routine, universal computation (see e.g. Garc\'ia-Selfa--Tornero \cite[\S2]{gst} for a derivation).
\end{proof}

We now turn to $m=4$.  To set things up, suppose that $E$ has a nontrivial $2$-torsion point $T \in E(F)$.  Writing $T = (-b,0)$, we have a model
\begin{equation} \label{eqn:EyaB}
E \colon y^2 = x^3 + Ax + b^3 + Ab.
\end{equation}

\begin{lem} \label{discsquarelem}
Let $R$ be a $2$-division point of $T$ on $E$, i.e., $2R=T$.  Then the following are equivalent:
\begin{enumroman}
\item $x(R) \in F$;
\item $3b^2 + A \in F^{\times 2}$; and
\item $E$ admits an $F$-rational cyclic $4$-isogeny whose kernel contains $T$.
\end{enumroman}
\end{lem}

\begin{proof}
The 2-division points of $T$ form a torsor under $E[2]$ and there are two $x$-coordinates.  Computing with the group law on a universal curve, the minimal polynomial of the $x$-coordinates is exactly
$$
x(R)^2 + 2bx(R) - (A+2b^2).
$$
Thus, $x(R) \in F$ if and only if the discriminant $12b^2+4A$ is a non-zero square in $F$, showing (i) $\Leftrightarrow$ (ii).  

For (i) $\Rightarrow$ (iii), if there exists $R$ with $x(R) \in F$, then the subgroup $\langle R \rangle = \lbrace 0, R, T, 3R \rbrace$ is stable under $\Gal(\overline{F}/F)$ since 
$$
3R = -R = (x(R), -y(R)).
$$  
For (iii) $\Rightarrow$ (i), if $\langle R\rangle$ is Galois stable, then for all $\sigma \in \Gal(\overline{F}/F)$ we have $\sigma(R) = \pm R$ so $\sigma(x(R)) = x(R)$, whence  $x(R) \in F$.
\end{proof}

\begin{prop} \label{4weierstrass}
The elliptic curve $E$ locally has a subgroup of order $4$ if and only if at least one of the following statements hold:
\begin{enumroman}
\item $E(F)[2] \simeq (\Z/2)^2$, or
\item $E$ has a cyclic $4$-isogeny defined over $F$.
\end{enumroman}

Moreover, in case \textup{(ii)}, there exist $a,b \in F$ such that $E$ is defined by
\[ y^2 =x^3 + (a^2 - 3b^2)x  + a^2b - 2b^3 \]
and the following statements hold:
\begin{itemize}
\item $E(F)[2] \simeq (\Z/2)^2$ if and only if $9b^2-4a^2 \in F^{\times 2}$, and
\item $E(F)[4] \not\subseteq E(F)[2]$ if and only if $2a-3b \in F^{\times 2}$ or $-2a-3b \in F^{\times 2}$. 
\end{itemize}
\end{prop}

\begin{proof}
An elliptic curve $E/F$ admits a rational cyclic 4-isogeny if and only if it has a Galois-stable cyclic subgroup of order 4; by stability, $E$ has a rational point of order 2 contained in the cyclic group.  Then the 2-adic representation of $E$ lies in the group
$$
\begin{pmatrix}
1 + 2\Z_2 & \Z_2 \\ 4 \Z_2 & 1+ 2\Z_2
\end{pmatrix}  \equiv \begin{pmatrix} * & * \\ 0 & * \end{pmatrix} \pmod{4}.
$$
Clearly, $\det(1-g) \equiv 0 \pmod{4}$ for all elements of this group and so $E$ has a local subgroup of order 4.  Conversely, if $\det(g-1) \equiv 0 \pmod{4}$ for all $g \in \im \rho_{E,2}$, but $E$ does not have full rational 2-torsion, then it will have one point of order 2 defined over $F$.  Then any non-trivial $g \in \im \rho_2$ reduces modulo 2 to $\left( \begin{smallmatrix} 1 & 1  \\ 0 & 1 \end{smallmatrix} \right)$ or $\left( \begin{smallmatrix} 1 & 0 \\0 & 1 \end{smallmatrix} \right)$ and so can be written in the form
$$
\begin{pmatrix}
1+2\alpha & \beta \\ 2 \gamma & 1+2\delta
\end{pmatrix}, \text{ or } 
\begin{pmatrix}
1+2\alpha' & 2\beta' \\ 2\gamma' & 1+2\delta'
\end{pmatrix},
$$
respectively.  

In the first case, $\det(1-g) \equiv 0\pmod{4}$ implies $2 \mid \beta \gamma$. But then $2 \mid \gamma$ (or else $E(F)[2]  = \Z/2 \times \Z/2$) and so the mod 4 representation on these elements has the shape $\left(\begin{smallmatrix} * & * \\ 0 & * \end{smallmatrix} \right)$.  If $g$ is of the second form, then multiply by a matrix $h$ of the first form and compute
$$
\det(1 - gh) \equiv 2(\gamma + \gamma')\beta \pmod{4}.
$$
Since $2\mid \gamma$, we must have $2 \mid \beta \gamma'$.  And similar to the first case we conclude that $2 \mid \gamma'$.  Thus the lower-triangular entry of any element of $\im \rho_{E,2}$ is divisible by 4 and so $E$ admits a rational cyclic 4-isogeny.

Now suppose we are in Case \textup{(ii)}.  In light of Lemma \ref{discsquarelem}, let $a \in F$ be such that $a^2 = 3b^2 + A$, so that 
$$
E:~y^2 =x^3 + (a^2 - 3b^2)x  + a^2b - 2b^3.
$$
Fix square-roots $\sqrt{\pm2a - 3b} \in \Fbar$.  Then $E$ visibly has two $F$-rational cyclic 4-isogenies with kernels
$$
\langle(a-b, a\sqrt{2a-3b})\rangle,~\text{and}~\langle(-a-b, a\sqrt{-2a-3b})\rangle,
$$
respectively.  Doubling either generator results in the marked 2-torsion point $T$; the other 2-torsion points are then 
$$
(b/2 \pm \sqrt{9b^2-4a^2}/2, 0).
$$
The splitting field of the preimages of $T$ under duplication is then a biquadratic extension of $F$ with intermediate extensions
$$ 
F(\sqrt{2a-3b}),\ F(\sqrt{-2a-3b}),\ F(\sqrt{9b^2-4a^2}).
$$
These quadratic extensions are nontrivial exactly under the conditions stated in Proposition \ref{4weierstrass}.  
\end{proof}

\subsection{The Probability \texorpdfstring{$P_3$}{P3}}  

By Lemma \ref{E3weierstrass}, all curves in $\scrE_{3?}$ either have global point of order 3 or are a quadratic twist by $-3$ of one that does.  These are modeled by the Weierstrass equations
\begin{align} \label{3weier}
y^2 = x^3 + u^2(6ab+27a^4)x + u^3(b^2-27a^6), 
\end{align}
where $u =1$ means the curve has a 3-torsion point and $u=-3$ is its quadratic twist.  

Denote by $R_{u}(H)$ the region (\ref{region}) attached to the elliptic curve (\ref{3weier}).  Applying the Principle of Lipschitz we see that
\begin{align} \label{lip_prin}
\Area R_u(H) = \Area R_u(1) H^{1/3} + O(H^{1/4}).
\end{align}
In particular, observe that 
\begin{align} \label{9factor}
\Area R_1(1) = 9 \Area R_{-3}(1).
\end{align}

\begin{rmk}
As long as $\#G \geq 5$, we have shown that Lipschitz asymptotics give a growth term of $2/d(G)$ and error of $1/d(G)$.  The degrees of the polynomials $A(a,b)$ and $B(a,b)$ are not large enough to ensure these asymptotics when $\#G =3$ or 4, so we estimate the order of growth of the error term ``by hand'' (using the results previously obtained by Harron-Snowden).
\end{rmk}

We are now ready to compute $P_3$.

\begin{prop} \label{prop:onehalfp3}
We have $P_3 = 1/2$.
\end{prop}

\begin{proof}
Appealing to the notation of (\ref{prob_definition}), we write $c=c(G_3(1;1,0))$ and $c'=c(G_3(1;0,0))$ and find
\[ c H^{1/3} + O(H^{1/4}) \text{ and } c' H^{1/3} + O(H^{1/4}) \]
for the number of minimal elliptic curves of height at most $H$ with a global 3-torsion subgroup and the number of quadratic twists, respectively; note the exponents come from the Lipschitz estimate of (\ref{lip_prin}).  It remains to compute $c$ and $c'$ exactly. 

For every prime $q \ne 3$, we have 
$$
q^4 \mid (6ab+27a^4) \text{ and } q^6 \mid (b^2-27a^6)
$$
if and only if
$$
q^4 \mid 9(6ab+27a^4) \text{ and } q^6 \mid -27(b^2-27a^6).
$$
Therefore, sieving out non-minimal equations away from $q=3$ has no effect on the ratio of the growth constants.

The pairs $(a,b)$ such that 
$$
3^4 \mid (6ab+27a^4) \text{ and } 3^6 \mid (b^2-27a^6)
$$
have $a \equiv 0\pmod{3}$ and $b \equiv 0 \pmod{27}$, which accounts for a proportion of $1/81$ of the pairs.  

For the twists, observe that 
$$
9(6ab+27a^4) \text{ and } -27(b^2-27a^6)
$$
are integral if and only if $a,b \in (1/3)\Z$.  Among those pairs, similar reasoning shows that 1/81 yield non-minimal equations.

Taking $a,b \in (1/3)\Z$ scales the area of $R_{-3}(H)$ by $9$, whence, by \eqref{9factor} the number of integral equations parameterizing 3-torsion and local 3-torsion is the same.  Sieving out 1/81 of the pairs from each count does not affect the ratio and so the proportions are equal.
\end{proof}

\begin{remark} \label{rmk:P3approx}
We confirm Proposition \ref{prop:onehalfp3} experimentally: in a naive way, we compute
\[ \frac{\#\{E \in \scrE_{\leq 10^{12}} : 3 \mid \#E(\Q)\sbtors\}}{\#(\scrE_{3?} \cap \scrE_{\leq 10^{12}})} = \frac{3808}{7578} \approx 0.503.  \]
\end{remark}

\subsection{The Probability \texorpdfstring{$P_4$}{P4}}  The strategy here is similar, but we will need to do more computation to get the growth constants exactly.  (The difference between this case and $P_3$ is that the Weierstrass models of the curves in $\scrE_{4}$ are not simply quadratic twists of each other and, moreover,  to argue how the shapes of the regions are transformed by cyclic isogenies is at least as difficult as computing the areas by calculus.)

First, we reduce our work by observing from Harron--Snowden \cite[Theorem 1.1]{hs} that the number of curves up to height $H$ with a rational $\Z/4$-torsion subgroup is $\asymp H^{1/4}$ and curves with full 2-torsion are $\asymp H^{1/3}$.  This shows that as $H \to \infty$, curves with full 2-torsion dominate curves with a 4-torsion point in $\scrE_{4?}$ and so the latter will not contribute to the probability $P_4$. For completeness, however, we record the quantities $d(\Z/4)$ and $e(\Z/6)$ in the following Proposition and fill in the entry for $\Z/4$ in Table \ref{tab:yup}. Because we do not need the growth constant, we are content to sketch a proof.

\begin{prop} \label{Z4constants}
The number $N_{\Z/4}(H)$ of elliptic curves over $\Q$ of height $\leq H$ with a rational point of order 4 is given by
$$
N_{\Z/4}(H) = cH^{1/4} + O(H^{1/6}),
$$
for an explicitly computable constant $c$.
\end{prop}

\begin{proof}
By \cite[$\Sigma_4$, p.~93]{gst} elliptic curves over $\Q$ with a rational point of order 4 are parameterized by
$$
y^2 = x^3  -27(16t^2 + 16t + 1)x -54(64t^3 - 120t^2 - 24t - 1).
$$
Homogenizing and clearing denominators across the Weierstrass equation, shows that 
 the number of integral equations is roughly given by the number of integral points in the compact region of $\R^2$ defined by
\begin{equation}
R(H) \colonequals \lbrace (a,b) : 4 \abs{A(a,b)}^3 \leq H \ \text{and} \ 27\abs{B(a,b)}^2 \leq H \rbrace
\end{equation}
where
\begin{equation}
\begin{aligned}
A(a,b) &\colonequals 27(16a^2 + 16ab^2 + b^4) \\
B(a,b) &\colonequals 54(64a^3-120a^2b^2 - 24ab^4 - b^6).
\end{aligned}
\end{equation}
(Here, ``roughly'' means that $A(a,b)$ and $B(a,b)$ are integral if and only if $(a,b) \in \left(\frac{1}{6}\Z\right) \times \Z$; this can be deduced from congruences.  We will not pursue a finer estimate than this because we do not seek an explicit growth constant.)

The compactness of $R(H)$ allows for a Lipschitz analysis.  A homogeneity argument with the Weierstrass coefficients (scale $a$ by $H^{1/6}a$ and $b$ by $H^{1/12}b$) shows immediately that $\Area(R(H)) = \Area(R(1))H^{1/4}$ and $O(\len (\bd (R(H)))) = O(H^{1/6})$.  The boundary of $R(H)$ is rectifiable (given by polynomials) and so the area of $R(1)$ is calculable.  The constant $c$ is $\area(R(1))$ scaled by $1/r(\Z/4)$ and a sieve factor, both of which are finite calculations. 
\end{proof}

For $G = \Z/2 \times \Z/2$, our next goal is to show that the number of isomorphism classes $N_G(H)$ of elliptic curves with global torsion subgroup $G$ of height $\leq H$ is given by 
\begin{align} \label{growth_rate}
N_G(H) = c(G)H^{1/d(G)} + O(H^{1/e(G)}).
\end{align}
Thus, $P_4$ will be given as a weighted ratio of the constant $c(\Z/2 \times \Z/2)$ and the corresponding constant for curves admitting a cyclic 4-isogeny.  We first work out the details for the group $\Z/2\times \Z/2$ in the following Proposition, which contributes to the data in Table \ref{tab:yup}.  After this, we count curves admitting a cyclic 4-isogeny in Proposition \ref{cyclic_isogeny_constants}.  From there, it is then a simple matter to fit the pieces together to obtain an exact expression for $P_4$; this is Corollary \ref{4prob}.

\begin{prop} \label{full_2tors_constants}
For $G=\Z/2 \times \Z/2$, we have
\begin{align*}
c(G) &=  \frac{121\pi \sqrt{3}\sqrt[3]{2}}{360}\\
r(G) &= 6\\
1/d(G) &= 1/3 \\
1/e(G) &= 1/6.
\end{align*} 
\end{prop}

\begin{proof}
We start with a two-variable model parameterizing elliptic curves with full 2-torsion:
\begin{align} \label{full2torsion}
y^2 = x^3  -\frac{(a^2-ab+b^2)}{3}x -\frac{(a+b)(2a-b)(a-2b)}{27},
\end{align}
identifying the polynomials $A(a,b)$ and $B(a,b)$ as
\begin{align*}
A(a,b) &= -(a^2-ab+b^2)/3\\
B(a,b) &= -(a+b)(2a-b)(a-2b)/27.
\end{align*}
It is routine to check that for all $H > 0$ we have the containment
$$
\lbrace (a,b)  : \abs{4 A(a,b)}^3 \leq H \rbrace \subseteq \lbrace (a,b)  : \abs{27B(a,b)}^2 \leq H \rbrace.
$$
(Briefly, rotate by $\pi/4$ so that it amounts to checking 
\begin{align} \label{rotated_ellipse}
4\left| \frac{a^2}{2} + \frac{b^2}{6} \right|^3 \leq H \Longrightarrow 27 \left( \frac{ba^2}{3\sqrt{2}} - \frac{b^3}{27\sqrt{2}} \right)^2 \leq H.
\end{align}
By symmetry and scaling, it suffices to show (\ref{rotated_ellipse}) holds for $a,b \geq 0$ and $H=1$, which is easily verified.)

We therefore put 
\begin{align} \label{R4defn}
R_4(H) = \lbrace (a,b) \in \R \times \R:  4|A(a,b)|^3 \leq H \rbrace.
\end{align}
The constants $c(\Z/2\times \Z/2)$, $d(\Z/2\times \Z/2)$, $e(\Z/2\times \Z/2)$ of the Proposition will follow from asymptotic analysis of the elliptical region defined by (\ref{R4defn}).

By the homogeneity of $A(a,b)$ of degree 2, it follows from direct calculation that
$$
\area (R_4(H))  = \area (R_4(1))H^{1/3}.
$$
By the Principle of Lipschitz applied to the homogeneously expanding compact region $R_4(H)$, 
we get that the number of integral points in $R_4(H)$ is asymptotically
$$
\area (R_4(H))  +O(\len(\bd(R_4(H)))).
$$
Therefore, $1/d(\Z/2 \times \Z/2) = 1/3$.  The fact that $R_4(H)$ defines an ellipse centered at the origin with boundary equation
$$
x^2-xy+y^2 = 3\left(\frac{H}{4} \right)^{1/3},
$$
immediately shows that 
$$
\len(\bd(R_4(H))) = O(H^{1/6}).
$$
It remains to remove singular and sieve out non-minimal equations. The conclusion from the steps will be that $1/e(\Z/2 \times \Z/2) = 1/6$ and an explicit expression for $c(\Z/2 \times \Z/2)$. 

The singular equations of the form (\ref{full2torsion}) have discriminant 0:
$$
4A(a,b)^3+27B(a,b)^2 = -a^2b^2(a-b)^2 =0,
$$
and by algebraic substitution we see that the number of singular equations up to height $H$ is $O(H^{1/6})$.  Therefore, the singular equations can be absorbed into the error term and we can now conclude that $1/e(\Z/2 \times \Z/2) = 1/6$.  

The points of $R_4(H)$ give a 6-fold overcount of models of the form (\ref{full2torsion}) because the points 
$$
\{(a,b),(b,a), (-a,b-a), (b-a,-a),(a-b,-b),(-b,a-b) \}
$$
each give rise to the identical Weierstrass equation with height $\leq H$; this shows $r(G) = 6$, as claimed. We also note that if both $A(a,b)$ and $B(a,b)$ are integers, then both $a$ and $b$ are integers, which is routinely verified by congruences, occurs for 1/3 of all integral pairs $(a,b) \in \Z\times \Z$. Therefore, 
\begin{align} \label{18}
\frac{\area R_4(1)}{18}
\end{align}
is the growth constant for non-singular, integral equations of the form (\ref{full2torsion}) of height $\leq H$.  It remains to sieve non-minimal equations.  We omit the routine computation, which is similar to the ones detailed in Section \ref{sect57} above, and simply observe that 
\begin{enumalph}
\item If $p \ne 3$, then 
$$
p^4 \mid A(a,b) \text{  and  } p^6 \mid B(a,b)
$$ 
if and only if $a \equiv b \equiv 0 \pmod{p^2}$. 
\item If $p=3$, then 
$$
3^4 \mid A(a,b) \text{ and } 3^6 \mid B(a,b)
$$
if and only if $(a,b) \equiv (0,0)$ or $(9,18)$ or $(18,9) \pmod{27}$.
\end{enumalph}

Thus, if $p\ne 3$ then $1/p^4$ of the equations are non-minimal at $p$.  If $p=3$, then $1/3^5$ equations are non-minimal.  Putting together (\ref{18}), this sieve,  and the area of the ellipse $R_4(1)$, we see that
$$
\frac{c(\Z/2 \times \Z/2)}{r(\Z/2 \times \Z/2)} = \frac{1}{18} \cdot \left(\frac{1-\frac{1}{3^5}}{1-\frac{1}{3^4}} \right) \frac{ \pi \sqrt{3}\sqrt[3]{2}}{\zeta(4)}= \frac{121\pi \sqrt{3}\sqrt[3]{2}}{2160\zeta(4)} \approx 0.355,
$$
which completes the proof.
\end{proof}

We now perform the analogous computation for curves admitting a cyclic 4-isogeny.  

\begin{prop} \label{cyclic_isogeny_constants}
The number $N(H)$ of elliptic curves over $\Q$ of height $\leq H$ admitting a cyclic $4$-isogeny is given by
$$
N(H) = cH^{1/d} + O(H^{1/e}),
$$
where 
\begin{align*}
c &= \frac{\area R_4'(1)}{2\zeta(4)} \approx 0.9574 \\
1/d &= 1/3 \\
1/e &= 1/6,
\end{align*} 
with the exact value of $c$ given in Lemma \textup{\ref{isogenyarealem}}.
\end{prop}

\begin{proof}
Appealing to Proposition \ref{4weierstrass} we define the region 
\begin{align*}
R_4'(H) = \lbrace (a,b) \in \R \times \R : 4\abs{a^2 - 3b^2}^3 \leq H \text{ and }  27\abs{a^2b - 2b^3}^2 \leq H\rbrace
\end{align*}
parameterizing curves of height $\leq H$ that admit a cyclic 4-isogeny.   We follow the same approach as in Proposition \ref{full_2tors_constants} to compute $1/d$, and $1/e$.  We separate the calculation of $c$ into a separate lemma following this Proposition. 

It follows from homogeneity of $A(a,b)$ and $B(a,b)$ that $\area(R_4'(H)) = \area(R_4'(1))H^{1/3}$ and by applying the Principle of Lipschitz we get $1/d = 1/3$.  By inspection on the degrees of $A(a,b)$ and $B(a,b)$, and using the fact that $A$ and $B$ are polynomials (so rectifiable) we see that $\len \bd (R_4'(H)) = O(H^{1/6})$.  

Next, we calculate the discriminant 
$$
4A(a,b)^3+27B(a,b)^2 = a^4(4a^2 - 9b^2)
$$
and see that the number of singular equations is $O(H^{1/6})$.  These singular equations can be absorbed into the Lipschitz error and we conclude that $1/e = 1/6$.

It remains to obtain $c$.  The region $R_4'(1)$ has polynomial boundary and its area can be computed by calculus (see the statement of Lemma \ref{isogenyarealem} immediately following this proof for an exact value of this area and numerical approximation).  We compute that $r(G)=2$.  

It is straightforward to verify that for every prime $p$, we have $p^4 \mid (a^2-3b^2)$ and $p^6 \mid (a^2b-2b^3)$ if and only if $a \equiv b \equiv 0 \pmod{p^2}$.  Sieving, we scale by $\zeta(4)^{-1}$.  
Altogether, we arrive at the growth constant
$$
c = \frac{1}{2} \cdot \frac{1}{\zeta(4)} \area (R_4'(1)) \approx 0.9574
$$
as claimed.
% where the numerical approximation is given by the result of Lemma \ref{isogenyarealem}.
\end{proof}

\begin{lem} \label{isogenyarealem}
Let $u = 4^{-1/3}, v = 27^{-1/2}$, and define the polynomials $F_{\pm} \in \R[x]$ by
$$
F_{\pm}(x) = x^3 \pm ux-v.
$$
Let $\alpha_{\pm}$ denote the unique positive root of $F_{\pm}$ and set $\beta_{\pm}  = \sqrt{3\alpha_{\pm}^2 \pm u}$. Where it is defined, let $I(p,q)$ denote the integral
$$
I(p,q) = \int_p^q \sqrt{\frac{2y^3+v}{y}}\,{\rm d}y.
$$
Then we have
\begin{align*}
\area(R_4'(1)) &= 4I(\alpha_+,\alpha_-) +2(\alpha_+\beta_+ - \alpha_-\beta_-) \\
&\qquad\qquad + \frac{2u}{\sqrt{3}} \log \left( \frac{(\sqrt{3} \alpha_+ + \beta_+)(\sqrt{3}\alpha_- + \beta_-)}{u} \right) \\
&\approx 2.072.
\end{align*}
\end{lem}

\begin{proof}
Straightforward calculation: for a bit more detail, see Pomerance--Schaefer \cite[\S 2]{ed_carl_arxiv}, where our area is $2i_4 \approx 2(1.036) \approx 2.072$.
\end{proof}

\begin{cor} \label{4prob}
We have 
\[ P_4 = \frac{121 \Area(R_4(1))}{121 \Area(R_4(1)) + 1080 \Area(R_4'(1))} \approx 0.270. \]
\end{cor}

\begin{proof}
Because both growth rates are $O(H^{1/3})$, we can express $P_4$ as the following ratio
$$
P_4 = \frac{c(\Z/2 \times \Z/2)}{c(\Z/2 \times \Z/2) + c}.
$$
The exact value and its approximations follow immediately from Propositions \ref{full_2tors_constants} and \ref{cyclic_isogeny_constants}.
\end{proof}

\begin{remark} \label{rmk:ed_carl}
Pomerance--Schaefer \cite{ed_carl_arxiv} count elliptic curves with Galois-stable cyclic subgroups of order 4 and obtain a finer estimate than our Proposition \ref{cyclic_isogeny_constants}, in the case where they count the number of curves with at least one pair of cyclic subgroups of order 4.  In that case they show 
$$
N(H)=c_{1}H^{1/3} + c_2H^{1/6}+O(H^{0.105}),
$$
where their $c_1$ is exactly our $c$ in Proposition \ref{cyclic_isogeny_constants}; they also compute the area of the same region that we do in Lemma \ref{isogenyarealem}.
\end{remark}

\begin{exm}
Returning to Example \ref{exm:isoggraph}, we have shown that $100\%$ of elliptic curves $E \in \scrE_{4?}$ are isogenous to an elliptic curve with full $2$-torsion (and no further torsion structure); the isogeny class of such curves have isogeny graph which is a tree with three leaves attached to a central root, for example the isogeny class with LMFDB \cite{lmfdb} label \href{http://www.lmfdb.org/EllipticCurve/Q/350/b/}{\textsf{350.b}}.
\end{exm}

\begin{remark} \label{rmk:P4approx}
We now give some experimental confirmation of Corollary \ref{4prob}.  Enumerating curves in a naive way, among the curves $E \in \scrE_{4?} \cap \scrE_{\leq 10^{13}}$ we count:
\begin{center}
\begin{tabular}{c|c} 
$E(\Q)\sbtors[2^\infty]$ & count \\
\hline \hline
$\Z/2$\rule{0pt}{2.5ex} & 20612 \\
$\Z/2 \times \Z/2$ & 8126 \\
$\Z/2 \times \Z/4$ & 8 \\
$\Z/4$ & 1382 \\
$\Z/8$ & 2
\end{tabular}
\end{center}
(It appears that the elliptic curve of smallest height with $\#E(\Q)\sbtors[2^\infty] \simeq \Z/2 \times \Z/8$ is the elliptic curve \href{http://www.lmfdb.org/EllipticCurve/Q/210/e/6}{\textsf{210.e6}} with height $\approx 10^{19.03}$.)  The curves with a rational $4$-torsion point are, according to the above, a lower-order term---but this is not so totally apparent in the range of our data!  So we estimate the probability by
\begin{equation}
\begin{aligned}
&\frac{\#\{E \in \scrE_{\leq 10^{13}} : \#E(\Q)\sbtors[2^\infty] \simeq \Z/2 \times \Z/2\}}{\#\{E \in \scrE_{4?} \cap \scrE_{\leq 10^{13}} : \#E(\Q)\sbtors[2^\infty] \leq \Z/2 \times \Z/2\}} \\
&\qquad\qquad = \frac{8126}{20612+8126} = \frac{8126}{28738} \approx 0.283
\end{aligned}
\end{equation}
which matches Corollary \ref{4prob} reasonably well.
\end{remark}

\begin{rmk} \label{rmk:naiveheight}
Alternatively, one can order the elliptic curves by naive height
\[ \height'(E) \colonequals \max(\abs{A^3},\abs{B^2}) \]
(without the scaling factors $4,27$) and ask how the explicit probabilities are affected.  This does not affect $P_3$, since the ratio of the areas of the regions $R_1(1)$ and $R_{-3}(1)$ is preserved.  However, in the case of $P_4$, the area of the elliptical region is $2\sqrt{3}\pi$ and the area of the region $R_4'(1)$ is given explicitly by 
\begin{align*} 
&4 \cdot \left(\frac{(\alpha_+\beta_+ - \alpha_-\beta_-)}{2} + \frac{\log\left((\beta_+ + \sqrt{3}\alpha_+)(\beta_- + \sqrt{3}\alpha_-)\right)}{2\sqrt{3}} \right. \\
&\qquad\qquad \left. + I(\alpha_+,\alpha_-) \right) \approx 4.019,
\end{align*}
where $\alpha_\pm$ is the real root of $z^3 \pm z-1$, $\beta_\pm = \sqrt{3\alpha_\pm^2 \pm1}$, and 
$$
I(p,q) = \int_p^q \sqrt{\frac{2z^3 + 1}{z}} {\rm d} z.
$$
No other adjustments to the growth constants are required.  Thus, the effect of ordering by $\height'$ versus $\height$ gives $P_4 \approx 0.233$.
\end{rmk}


\bigskip 

\begin{thebibliography}{xx}

\bibitem{alvaro} 
\textsc{G. Chiloyan} and \textsc{\'{A}. Lozano-Robledo}, \textit{A classification of isogeny-torsion graphs of elliptic curves over $\Q$}, Trans. London Math. Soc. \textbf{8} (2021), no.~1, 1--34. 
\bibitem{assaf}
\textsc{E. Assaf}, \textit{Computing classical modular forms for arbitrary congruence subgroups}, 2021, accepted to Simons Symp.
\bibitem{Baran} 
\textsc{B. Baran}, \textit{Normalizers of non-split Cartan subgroups, modular curves, and the class number one problem}, J.\ Number Theory \textbf{130} (2010), no.~12, 2753--2772.
\bibitem{BS}
\textsc{M. Bhargava} and \textsc{A. Shankar}, \textit{Binary quartic forms having bounded invariants, and the boundedness of the average rank of elliptic curves}, Ann.\ of Math.\ (2) \textbf{181} (2015), no.\ 1, 191--242.
\bibitem{boggess-sankar}
\textsc{B. Boggess} and \textsc{S. Sankar}, \textit{Counting elliptic curves with a rational $N$-isogeny for small $N$}, preprint, 2020, \texttt{arXiv:2009.05223}.
\bibitem{magma}
\textsc{W. Bosma}, \textsc{J. Cannon}, and \textsc{C. Playoust}, \textit{The Magma algebra system.\ I.\ The user language}, J.\ Symbolic Comput.\ \textbf{24} (3--4), 1997, 235--265.
\bibitem{bruinnaj}
\textsc{P. Bruin} and \textsc{F. Najman}, \textit{Counting elliptic curves with prescribed level structures over number fields}, preprint, 2021, \texttt{arXiv:2008.05280}.
\bibitem{cj}
\textsc{P.\ Cho} and \textsc{K.\ Jeong}, \textit{Probabilistic behaviors of elliptic curves with torsion points}, preprint, 2020, \texttt{arXiv:2005.06862}.
\bibitem{CLO}
\textsc{D. A. Cox}, \textsc{J. Little}, and \textsc{D. O'Shea}, \textit{Using algebraic geometry}, 2nd.\ ed., Grad.\ Texts in Math., vol.~185, Springer, New York, 2005.
\bibitem{polysonline}
\textsc{J. Cullinan} and \textsc{J. Voight}, \textit{Universal polynomials for $m$-full torsion groups}, 2020, [Online; available at \url{http://math.dartmouth.edu/~jvoight/code/compute_universal.m}].
\bibitem{davenport} 
\textsc{H. Davenport}, \textit{On a principle of Lipschitz}, J.\ London Math.\ Soc.\ \textbf{26} (1951), 179--183; Corrigendum, J.\ London Math.\ Soc.\ \textbf{39} (1964), 580.
\bibitem{DR} 
\textsc{P. Deligne} and \textsc{M. Rapoport}, \textit{Les sch\'emas de modules de courbes elliptiques}, Modular functions of one variable, II (Proc. Internat. Summer School, Univ. Antwerp, Antwerp, 1972), Lecture Notes in Math., vol.~349, Springer, Berlin, 1973, 143--316.
\bibitem{DS} 
\textsc{F. Diamond} and \textsc{J. Shurman}, \textit{A first course in modular forms}, Grad.\ Texts in Math., vol.~228, Springer-Verlag, New York, 2005. 
\bibitem{eszb}
\textsc{J. Ellenberg}, \textsc{M. Satriano}, and \textsc{D. Zureick-Brown}, \textit{Heights on stacks and a generalized Batyrev--Manin--Malle conjecture}, preprint, 2021, \texttt{arXiv:2106.11340v1}.
\bibitem{gst} 
\textsc{I. Garc\'ia-Selfa} and \textsc{J. M. Tornero}, \textit{A complete Diophantine characterization of the rational torsion of an elliptic curve}, Acta Math.\ Sin.\ (Engl.\ Ser.)\ \textbf{28} (2012), no.~1, 83--96.
\bibitem{greenberg}
\textsc{R. Greenberg}, \textit{The image of Galois representations attached to elliptic curves with an isogeny}, Amer.\ J.\ Math.\ \textbf{134} (2012), no.~5, 1167--1196. 
\bibitem{hs} 
\textsc{R. Harron} and \textsc{A. Snowden}, \textit{Counting elliptic curves with prescribed torsion}, J.~Reine Angew.\ Math.\ \textbf{729} (2017), 151--170.
\bibitem{huxley}
\textsc{M.N. Huxley}, \textit{Exponential sums and lattice points III}, Proc.\ London Math.\ Soc.\ \textbf{87} (2003), no.~3, 591--609.
\bibitem{katz} 
\textsc{N. M.~Katz}, \textit{Galois properties of torsion points on abelian varieties}, Inv.\ Math.\ \textbf{62} (1981), 481--502.
\bibitem{KM} 
\textsc{N. M.~Katz} and \textsc{B. Mazur}, \textit{Arithmetic moduli of elliptic curves}, Annals of Math.\ Studies, vol.~108, Princeton University Press, Princeton, NJ, 1985.
\bibitem{lmfdb} 
\textsc{The LMFDB Collaboration}, The L-functions and Modular Forms Database, \url{http://www.lmfdb.org}, 2020, [Online; accessed 28 June 2020].
\bibitem{kra} 
\textsc{I. Kra}, \textit{On lifting Kleinian groups to $\SL(2,\C)$}, Differential geometry and complex analysis, Springer, Berlin, 1985, 181--193.
\bibitem{mazur} 
\textsc{B. Mazur}, \textit{Modular curves and the Eisenstein ideal}, Inst.\ Hautes Etudes Sci.\ Publ.\ Math., no.~47, 1977, 33--186.
\bibitem{ppj}
\textsc{M. Pizzo}, \textsc{C. Pomerance}, and \textsc{J. Voight}, \emph{Counting elliptic curves with an isogeny of degree three}, Proc.\ Amer.\ Math.\ Soc.\ Ser.\ B \textbf{7} (2020), 28--42.
\bibitem{ed_carl_arxiv} 
\textsc{C. Pomerance} and \textsc{E. F.~Schaefer}, \textit{Elliptic curves with Galois-stable cyclic subgroups of order~4}, Res.\ Number Theory \textbf{7} (2021), no.~2, Paper No.~35. 
\bibitem{sebbar}
\textsc{A. Sebbar}, \textit{Classification of torsion-free genus zero congruence groups}, Proc.\ Amer.\ Math.\ Soc.\ \textbf{129} (2001), no.~9, 2517--2527.
\bibitem{serre:lmw} 
\textsc{J.-P. Serre}, \textit{Lectures on the Mordell-Weil theorem}, 3rd ed., Aspects of Math., Friedr. Vieweg \& Sohn, Braunschweig, 1997.
\bibitem{serre} 
\textsc{J.-P. Serre}, \textit{Abelian $\ell$-adic representations and elliptic curves}, Res. Notes Math., vol.~7, A K Peters, Ltd., Wellesley, MA, 1998.
\bibitem{Shimura}
\textsc{G. Shimura}, \textit{Introduction to the arithmetic theory of automorphic functions}, Publ.\ Math.\ Soc.\ of Japan, vol.~11, Kan\^o Memorial Lectures, 1, Princeton University Press, Princeton, NJ, 1994.
\bibitem{SZ} 
\textsc{A. V.~Sutherland} and \textsc{D. Zywina}, \textit{Modular curves of prime-power level with infinitely many rational points}, Algebra Number Theory \textbf{11} (2017), no.~5, 1199--1229. 
\bibitem{velu} 
\textsc{J. V\'elu}, \textit{Isog\'enies entre courbes elliptiques}, C.\ R.\ Acad.\ Sci.\ Paris S\'er. A-B \textbf{273} (1971), A238--A241.
\bibitem{vzb} 
\textsc{J. Voight} and \textsc{D. Zureick-Brown}, \textit{The canonical ring of a stacky curve}, to appear in Mem. Amer. Math. Soc.
\bibitem{zywina} 
\textsc{D. Zywina}, \textit{Possible indices for the Galois image of elliptic curves over $\Q$}, preprint, 2015, \texttt{arXiv:1508.07663}.
\end{thebibliography}
\end{document}